\documentclass[11pt,reqno]{amsart}

\usepackage[margin=1in]{geometry}
\usepackage{amsmath,amssymb,mathtools}
\usepackage{aliascnt}
\usepackage{microtype}
\usepackage{tikz}
\usetikzlibrary{decorations.pathreplacing}
\usepackage[hidelinks]{hyperref}
\usepackage[nameinlink,noabbrev]{cleveref}
\usepackage{cite}

\numberwithin{equation}{section}

\newtheorem{theorem}{Theorem}[section]
\newaliascnt{proposition}{theorem}
\newtheorem{proposition}[proposition]{Proposition}
\aliascntresetthe{proposition}
\newaliascnt{lemma}{theorem}
\newtheorem{lemma}[lemma]{Lemma}
\aliascntresetthe{lemma}
\newaliascnt{corollary}{theorem}
\newtheorem{corollary}[corollary]{Corollary}
\aliascntresetthe{corollary}
\theoremstyle{remark}
\newaliascnt{remark}{theorem}
\newtheorem{remark}[remark]{Remark}
\aliascntresetthe{remark}

\newcommand{\R}{\mathbb R}
\newcommand{\E}{\mathbb E}
\renewcommand{\P}{\mathbb P}
\newcommand{\Var}{\operatorname{Var}}
\newcommand{\Cov}{\operatorname{Cov}}
\newcommand{\supp}{\operatorname{supp}}
\newcommand{\sech}{\operatorname{sech}}
\newcommand{\dd}{\,\mathrm d}
\newcommand{\e}{\mathrm e}
\newcommand{\1}{\mathbf 1}

\title[Full RSB in the SK model]{Full replica symmetry breaking in the Sherrington--Kirkpatrick model}
\author{Patrick Lopatto}
\date{\today}

\begin{document}

\begin{abstract}
We prove that, at zero external field and for every inverse temperature
$\beta>1$, the Parisi measure of the Sherrington--Kirkpatrick model is
supported on an interval $[0,q_\beta]$, with a smooth density on $(0,q_\beta)$ and an atom at $q_\beta$. The primary difficulty is to exclude gaps in
the support. On such a gap, the optimality conditions for the Parisi measure would force the second derivative of the associated self-consistency function $\Gamma$ to cross zero downward. We establish that it can only cross upward by showing that at a zero of $\Gamma''$, the
derivative $\Gamma'''$ is a positive linear combination of two nonnegative covariances, one of which is strictly positive.
\end{abstract}

\maketitle

\section{Introduction}

We consider the Sherrington--Kirkpatrick spin glass model at inverse temperature $\beta >0$, which is defined by the Hamiltonian
\begin{equation}\label{eq:SK-Hamiltonian}
 H_N^\beta(\sigma)
 =
 \frac{\beta}{\sqrt{2N}}
 \sum_{i,j=1}^N g_{ij}\sigma_i\sigma_j,
 \qquad
 \sigma\in\{-1,1\}^N,
\end{equation}
where $(g_{ij})_{1\le i,j \le N}$ are independent standard Gaussian random variables. The limiting free energy $F(\beta)$ of this model was characterized by Talagrand \cite{Talagrand2006}, building on Guerra's upper bound \cite{Guerra2003}. His result can be formulated as follows (see \cite[Section 1]{AuffingerChenUnique}).
For a Borel probability measure $\mu$ on $[0,1]$, let $u_\mu$ be the
solution of the Parisi equation
\begin{equation}\label{eq:parisi-pde}
 \partial_su_\mu(s,x)
 +\frac{\beta^2}{2}\left(
 u_{\mu,xx}(s,x)+\mu([0,s])u_{\mu,x}(s,x)^2
 \right)=0,
 \qquad
 u_\mu(1,x)=\log\cosh x.
\end{equation}
Define the Parisi functional
\begin{equation}\label{eq:parisi-functional}
 \mathcal P_\beta(\mu)
 =\log2+u_\mu(0,0)
 -\frac{\beta^2}{2}\int_0^1s\mu([0,s])\dd s.
\end{equation}
Then
\[
F(\beta)=\inf_{\mu\in\Pr([0,1])} \mathcal P_\beta (\mu),
\]
where $\Pr([0,1])$ is the set of Borel probability measures on $[0,1]$.
Additionally, the functional
\eqref{eq:parisi-functional} has a unique minimizer by
\cite[Corollary 1]{AuffingerChenUnique}, which we denote by $\mu_\beta$.

We are interested in determining the structure of $\mu_\beta$. For
$\beta\le 1$, it is known that $\mu_\beta$ is trivial,
$\mu_\beta=\delta_0$
\cite{AizenmanLebowitzRuelle1987,Talagrand2006,JagannathTobasco2017}. For $\beta>1$,
predictions from physics indicate that $\mu_\beta$ should contain an
interval in its support \cite{Parisi1979,Parisi1980}. It was shown in \cite{Toninelli2002}
that the replica-symmetric ansatz is not minimizing for any $\beta>1$,
and hence that $\mu_\beta$ is nontrivial throughout the low-temperature
phase. Further, it is known that
$0\in\operatorname{supp}\mu_\beta$ at every temperature and that
$\mu_\beta$ has a smooth density on every open interval contained in its
support \cite{AuffingerChenProperties}. For the SK model, the criterion of
\cite{AuffingerChenProperties} excludes one-step replica symmetry
breaking at sufficiently low temperature. The same work shows that, for
$1<\beta\le 3/2$, the maximal point of the support is an atom.

Recently, Zhou proved that, for $\beta>1$ sufficiently close to $1$,
there exists $q_\beta>0$ such that
\[
    \operatorname{supp}\mu_\beta=[0,q_\beta],
\]
and that $\mu_\beta$ has a smooth density on $[0,q_\beta)$ and an atom
at $q_\beta$ \cite{ZhouFRSB}. Before the present work, such a complete description of the Parisi measure
was not known outside the near-critical regime. It is also known that at zero temperature, the corresponding Parisi measure has infinite
support \cite{AuffingerChenZeng}. As a consequence, for every $k\ge 1$, there exists
$\beta_k<\infty$ such that $\mu_\beta$ has at least $k+1$ distinct
points in its support whenever $\beta>\beta_k$
\cite{AuffingerChenZeng}. For related results on the structure and phase transitions of Parisi measures in the presence of an external field and in more general $p$-spin
models, see
\cite{AuffingerChenProperties,JagannathTobasco2017,zhou2024gardner,LopattoAT,AronowLopatto2026,zhou2026gardner}.

Our contribution is the following theorem, which establishes full
replica symmetry breaking throughout the low-temperature phase.

\begin{theorem}\label{thm:main}
For every $\beta>1$, there exist $
 q_\beta\in(0,1)$, $c_\beta\in(0,1)$, and a nonnegative function $
 \rho_\beta\in C^\infty([0,q_\beta])$
such that
\begin{equation}\label{eq:main-decomposition}
 \mu_\beta(\dd s)
 =\rho_\beta(s)\1_{[0,q_\beta)}(s)\dd s+c_\beta\delta_{q_\beta}(\dd s).
\end{equation}
Moreover,
\begin{equation}\label{eq:main-conclusion}
 \supp\mu_\beta=[0,q_\beta].
\end{equation}
\end{theorem}

Here \(\rho_\beta\in C^\infty([0,q_\beta])\) means that
\(\rho_\beta\in C^\infty((0,q_\beta))\) and that every derivative of
\(\rho_\beta\) extends continuously to the closed interval
\([0,q_\beta]\). Building on an earlier version of
this work, H.-B.~Chen \cite{ChenHB2026} has shown that $q_\beta\to1$ as $\beta\to\infty$ and that the Parisi measure at zero temperature has support $[0,1)$.

Our methods also give a quantitative lower bound on the atom at the
maximal support point.

\begin{theorem}\label{thm:atom}
Let $\beta>1$, and let $q_\beta$ and $c_\beta$ be as in
\Cref{thm:main}. Then $m_\beta=1-c_\beta$
satisfies
\begin{equation}\label{eq:atom-quadratic}
 m_\beta\Bigl(4-m_\beta+\frac1{\beta^2q_\beta}\Bigr)\le2.
\end{equation}
Consequently,
\begin{equation}\label{eq:atom-bound}
 c_\beta>\max\bigl\{\sqrt2-1,\;1-2\beta^2q_\beta\bigr\}.
\end{equation}
\end{theorem}

In the physics literature, the expansions of Crisanti and Rizzo
\cite[(11)--(12)]{CrisantiRizzo2002} predict $m_\beta =2\varepsilon+O(\varepsilon^2)$ and
$q_\beta=\varepsilon+O(\varepsilon^2)$ as $\varepsilon=1-\beta^{-1}$ tends to zero, so the bound
$m_\beta<2\beta^2q_\beta$ from \eqref{eq:atom-bound} matches the predicted
behavior to leading order. At low temperature they find
$m_\beta \to0.548\pm0.005$ as $\beta\to\infty$
\cite[(35)]{CrisantiRizzo2002}, while \eqref{eq:atom-bound} gives
$m_\beta<2-\sqrt2\approx0.586$.

\subsection{Outline of the proof}\label{s:outline}

Let $\mu=\mu_\beta$, $\alpha(s)=\mu([0,s])$, $u=u_\mu$, and let $X$ be the
optimal diffusion \eqref{eq:optimal-diffusion} below. We first describe
the main observable used in the argument and then the four steps of the
proof.

\medskip\noindent\emph{The magnetization martingale.}
In the physical interpretation of the Parisi formula
\cite{SommersDupont1984,MezardParisiVirasoro1987}, $u(s,x)$ is the effective contribution to the free
energy of a single spin in a cavity field $x$ at level $s$ of the
ultrametric tree, and $u_x(s,x)\in(-1,1)$ is its magnetization. We consider
the magnetization along the optimal diffusion,
\[
 M_s=u_x(s,X_s),
\]
which is a martingale with $M_1=\tanh X_1$, so that
$M_s=\E[\tanh X_1\mid\mathcal F_s]$. 
By $u_{xx}>0$ and \eqref{eq:u-absolute-growth}, the map
$x\mapsto u_x(s,x)$ is strictly increasing from $-1$ to $1$, so every $B\in(-1,1)$ determines a unique field $x$. Let $\chi(s,B)$ denote the
susceptibility as a function of the magnetization at time $s$, defined by
$\chi(s,u_x(s,x))=u_{xx}(s,x)$. Then It\^o's
formula shows that
\begin{equation}\label{e:nodrift}
 \dd M_s=\beta\,\chi(s,M_s)\dd W_s,
\end{equation}
where $W$ is the Brownian motion driving $X$. The self-consistency function $\Gamma$ is defined as the variance of this martingale, and
its first derivative is the second moment of the susceptibility:
\begin{equation}\label{e:Gamma-magnetization}
 \Gamma(s)=\E M_s^2,
 \qquad
 \Gamma'(s)=\beta^2\E\chi(s,M_s)^2.
\end{equation}
On an interval where $\alpha$ is constant with $\alpha= m$, we also have
\begin{equation}\label{e:gamma-no-gap}
 \Gamma''(s)
 =2\beta^4\E\bigl[\chi(s,M_s)^2\,\Phi(s,M_s)\bigr],
 \qquad
 \Phi(s,B)=\frac12\chi_B(s,B)^2-m\chi(s,B).
\end{equation}
The optimality conditions
\cite{TalagrandParisiMeasures,AuffingerChenProperties,Chen2017,JagannathTobasco2017}
require that $\Gamma(q)=q$ and $\Gamma'(q)\le1$ for every $q\in\supp\mu$.

\medskip\noindent\emph{Step 1: Reduction to the shape of $\Gamma'$.}
Suppose that $(a,b)$ is a gap in the support with $a,b\in\supp\mu$, so
that $\alpha\equiv m$ on $[a,b)$ for some $m>0$. By the optimality
conditions, we have $\Gamma'(a)\le1$ and $\Gamma'(b)\le1$, while
$\Gamma(a)=a$ and $\Gamma(b)=b$ give
\[
 \frac1{b-a}\int_a^b\Gamma'(s)\dd s=1.
\]
If $\Gamma'$ were strictly decreasing and then strictly increasing on
$(a,b)$, its maximum over $[a,b]$ would be attained at an endpoint, so
that $\Gamma'\le1$ on $[a,b]$, and the displayed average would force
$\Gamma'\equiv1$, which is impossible for such a function
(\Cref{fig:gap}). We show that $\Gamma'$ has exactly this shape, with a possibly empty
decreasing or increasing part, on every interval where $\alpha$ is
constant and positive (\Cref{cor:Gamma-prime-shape}). This excludes the gap. The shape of
$\Gamma'$ is governed by the normalized average
\begin{equation}\label{e:f-intro}
 f(s)=\frac{\E[\chi(s,M_s)^2\Phi(s,M_s)]}{\E[\chi(s,M_s)^2]}
 =\frac{\Gamma''(s)}{2\beta^2\Gamma'(s)}
 =\frac{(\log\Gamma')'(s)}{2\beta^2},
\end{equation}
which has the sign of $\Gamma''(s)$ by \eqref{e:Gamma-magnetization} and
\eqref{e:gamma-no-gap}. It suffices to show that $f$ times a positive function of $s$ has strictly positive derivative, since then $f$ changes sign at most once and only from
negative to positive. In particular,
$\Gamma''(s)=0$ implies $\Gamma'''(s)>0$, which is the crossing property
of \Cref{prop:crossing}. A gap $[0,b)$ on which $\alpha\equiv0$ is
excluded more directly, since $\Gamma''(s)>0$ for $s\in(0,b)$ is incompatible with
$\Gamma(0)=0$, $\Gamma(b)=b$, and $\Gamma'(b)\le1$.

\begin{figure}[t]
\centering
\begin{tikzpicture}[x=5cm,y=2.4cm,>=stealth]
 \begin{scope}
  \draw[->] (-0.05,0) -- (1.12,0) node[right] {$s$};
  \draw[->] (0,-0.05) -- (0,1.3) node[above] {$\Gamma'$};
  \draw[dashed,gray] (0,1) -- (1.05,1);
  \node[left] at (0,1) {$1$};
  \draw[very thick,domain=0:1,samples=120]
   plot (\x,{0.85+0.236*sin(180*\x)});
  \fill (0,0.85) circle (1.3pt);
  \fill (1,0.85) circle (1.3pt);
  \node[below] at (0,0) {$a$};
  \node[below] at (1,0) {$b$};
  \node[font=\small] at (0.5,-0.32) {(a) forced by a gap};
 \end{scope}
 \begin{scope}[xshift=6.6cm]
  \draw[->] (-0.05,0) -- (1.12,0) node[right] {$s$};
  \draw[->] (0,-0.05) -- (0,1.3) node[above] {$\Gamma'$};
  \draw[dashed,gray] (0,1) -- (1.05,1);
  \node[left] at (0,1) {$1$};
  \draw[very thick,domain=0:1,samples=120]
   plot (\x,{0.9-0.4*sin(180*\x)+0.05*\x});
  \fill (0,0.9) circle (1.3pt);
  \fill (1,0.95) circle (1.3pt);
  \node[below] at (0,0) {$a$};
  \node[below] at (1,0) {$b$};
  \node[font=\small] at (0.5,-0.32) {(b) shape given by \Cref{cor:Gamma-prime-shape}};
 \end{scope}
\end{tikzpicture}
\caption{The derivative $\Gamma'$ on a hypothetical gap $(a,b)$ of the
support. (a) The optimality conditions $\Gamma'(a)\le1$, $\Gamma'(b)\le1$,
and $\Gamma(b)-\Gamma(a)=b-a$ force $\Gamma'$ to exceed $1$ at an interior
maximum unless $\Gamma'\equiv1$. (b) On every
interval where $\alpha$ is constant and positive, $\Gamma'$ is decreasing
and then increasing. With the same endpoint conditions, $\Gamma'$ stays
below $1$, and $\Gamma(b)-\Gamma(a)<b-a$.}
\label{fig:gap}
\end{figure}

\medskip\noindent\emph{Step 2: Covariance identity for $f'$.}
Since $M$ is a martingale, its distribution spreads out as time
increases, in the sense that expectations of convex functions are nondecreasing in time. On every interval where $\alpha\equiv m>0$, the function
$\Phi(s,\cdot)$ is negative at $B=0$ and increasing in $|B|$
(\Cref{lem:crossing-monotone}). Together with the
spreading of $M$, this suggests that $f$ increases in $s$. The martingale property alone does not control this weighted
average, since both $\Phi$ and $\chi^2$ evolve with time,
so we compute $f'$ directly. To analyze $f'$, we express the weighted average defining $f$
as an integral over the magnetization $B$. This coordinate is
convenient because the magnetization martingale $M$ has no drift when expressed in it 
(recall \eqref{e:nodrift}).  The
derivative of $f$ has two contributions: the change in the law of $M_s$ reweighted by $\chi(s, \cdot)^2$, and the change in $\Phi(s,B)$ at fixed $B$. We will show that one integration by parts and a cancellation
combine these contributions, up to a term proportional to $f$,
into two covariances of increasing functions.

We work on an interval in $(0,1)$ where $\alpha\equiv m>0$, in the rescaled time
$t=\beta^2s$, and we suppress the time argument. We set
\[
 z=-\frac12\chi_B,\qquad Q=-\chi\Bigl(\frac12\chi_{BB}+m\Bigr),\qquad
 H=z^2+m\chi-2Q,
\]
so that $\Phi=2z^2-m\chi$. We write $C=u_{xx}$, so that $\chi(s,u_x)=C$; these are the functions in \eqref{eq:z-Q-H} below. A subscript $x$
denotes a derivative in the field coordinate, so that
$\partial_x=\chi\,\partial_B$. At
fixed $B$, the susceptibility evolves by
\begin{equation}\label{e:chipde}
 \partial_t\chi=-\chi^2\Bigl(\frac12\chi_{BB}+m\Bigr)=Q\chi
\end{equation}
(\Cref{rem:slope}). Let $\eta$ be the density of $M_s$, and let $\rho$ be the
probability measure on $[0,1)$ with density proportional to
$r=\chi^2\eta$. Since $\chi$, $\Phi$, and $\eta$ are even in $B$,
$f=\E_\rho\Phi$. Since $M$ has no drift,
$\eta$ solves $\eta_t=\frac12(\chi^2\eta)_{BB}$, and since $\chi$ has logarithmic growth rate $Q$ at fixed $B$,
the weighted density solves
\begin{equation}\label{e:weighted-law}
 r_t=\frac12\chi^2r_{BB}+2Qr
\end{equation}
(\Cref{rem:weighted-law}). 

We compute $\dd f/\dd t$ in three steps. First, differentiating
the ratio defining $f$ gives
\[
 \frac{\dd f}{\dd t}=\E_\rho[\partial_t\Phi]+\Cov_\rho(\Phi,\partial_t\log r).
\]
Here, the partial time derivatives of $\Phi(t,B)$ and $r(t,B)$
are taken with $B$ fixed.
Second, we identify the two rates in this expression. By \eqref{e:chipde}, the change of
$\Phi$ at fixed $B$ is 
\begin{equation}\label{e:phit}
\partial_t\Phi=Q\Phi+zH_x.
\end{equation}
We define $N$ through the logarithmic derivative of the density of $\rho$ in the magnetization coordinate, $N+z=-\chi\,\partial_B\log r$. Then \eqref{e:weighted-law} and
the identity $z_x=Q+m\chi$ give
\[
 \partial_t\log r=\frac K2-Q-H,
 \qquad K=(N^2-N_x)+(z^2+z_x).
\]
Third, we use integration by parts to express the derivative term
$\E_\rho[zH_x]$ from \eqref{e:phit} in terms of  covariances. 
Set $\Psi=z(N+z)-Q$.  Integration by parts against the weighted
density gives
\[
 \E_\rho[zH_x]
 =\Cov_\rho(H,\Phi)+\Cov_\rho(H,\Psi)
\]
(\Cref{rem:weighted-law}; the boundary terms are shown to vanish in \Cref{lem:crossing-decay}).

We now substitute this formula into the expression for $f_t$ using \eqref{e:phit}.
The term $\Cov_\rho(H,\Phi)$ cancels the contribution
$-\Cov_\rho(\Phi,H)$  from the $-H$ term in
$\partial_t\log r=K/2-Q-H$.
Similarly, writing
\[
 \E_\rho[Q\Phi]
 =(\E_\rho Q)f+\Cov_\rho(Q,\Phi),
\]
the second term cancels $-\Cov_\rho(\Phi,Q)$ from the $\partial_t\log r$ term.
Only the mean term $(\E_\rho Q)f$ and two covariances remain:
\begin{equation}\label{e:positive-identity-intro}
 \frac{\dd f}{\dd t}-(\E_\rho Q)f
 =\frac12\Cov_\rho(\Phi,K)+\Cov_\rho(H,\Psi).
\end{equation}
See \eqref{eq:positive-representation} in the proof of
\Cref{prop:crossing} for the detailed justification.

Step~3 shows that $\Phi$, $K$, and $\Psi$ are strictly increasing
on $[0,1)$ and $H$ is nondecreasing. Thus both covariances in
\eqref{e:positive-identity-intro} are nonnegative, and the first
is positive because $\rho$ has a positive density. Writing $\rho_r$ for $\rho$ at time $r$ and fixing a base time
$t_0$ in the rescaled interval, the identity
therefore implies that
\[
 \exp\Bigl(-\int_{t_0}^t\E_{\rho_r}Q\dd r\Bigr)f(t)
\]
is strictly increasing. Since the exponential factor is positive,
$f$ can change sign at most once, and only from negative to
positive. This gives the shape of $\Gamma'$ required in Step~1
(\Cref{cor:Gamma-prime-shape}).

\medskip\noindent\emph{Step 3: The shape properties.}
For $s>0$, we write the density of $X_s$ as
\begin{equation}\label{e:factorization}
 p(s,x)=\e^{\alpha(s)u(s,x)}\cdot\e^{-V(s,x)},
\end{equation}
where $V$ is defined by this expression. The factorization is suggested
by the drift $\beta^2\alpha u_x=\beta^2\partial_x\log(\e^{\alpha u})$ of the
optimal diffusion. On an interval where $\alpha$ is constant, the first
factor solves the backward heat equation. Dividing $p$ by it cancels
the drift in the evolution equation of $p$, so that the second factor
solves the forward heat equation. For $s>0$ with $m=\alpha(s)>0$, we prove the following five
inequalities for all $x\ge0$ (equivalently, $B\ge0$).
The first three concern the backward factor, and the last two
concern the forward factor:
\[
 \text{(S1) } Q\ge0,\qquad \text{(S2) } z\ge mB,\qquad
 \text{(S3) } H_x\ge0,\qquad \text{(S4) } V_{xx}\ge\frac1{\beta^2s},\qquad
 \text{(S5) } V_{xxx}\le0.
\]
The inequality (S2) follows from (S1), and the other four give the
monotonicity required in Step~2. 
Since $\eta=p/C$ by the change of variables $\dd B=C\dd x$,
we have $r=pC$ when written in field coordinates.
By \eqref{e:factorization}, the definition of $N$ therefore reads
$N=V_x+z-mB$, so that $N_x=V_{xx}+Q$. \Cref{tab:shape} records which inequalities make each
function in \eqref{e:positive-identity-intro} monotone
(\Cref{lem:crossing-monotone}).

\begin{table}[t]
\centering
\begin{tabular}{lll}
 function & origin in Step~2 & monotone by\\
 \hline
 $\Phi=2z^2-mC$ & the integrand of $f$ & (S1)\\
 $K=(N^2-N_x)+(z^2+z_x)$ & contribution from the evolution of $r$ & (S1), (S4), (S5)\\
 $H=z^2+mC-2Q$ & change of $\Phi$ at fixed $B$ & (S3)\\
 $\Psi=z(N+z)-Q$ & the integration by parts & (S1), (S3), (S4)
\end{tabular}
\caption{The functions in the covariance identity
\eqref{e:positive-identity-intro} and the shape inequalities that make
them monotone in $x$ on $[0,\infty)$.}
\label{tab:shape}
\end{table}
We first prove the backward and forward inequalities for
finitely supported measures and then pass to arbitrary
measures by approximation.

\medskip
\emph{The backward factor.} At $s=1$, we have $\chi(1,B)=1-B^2$, for which the right side of
\eqref{e:chipde} is $(1-m)\chi^2$. Thus the terminal profile is
stationary when $m=1$, while lowering $m$ makes it decrease in
backward time; each atom lowers $m$ and strengthens this decrease.
\Cref{prop:backward} shows that throughout the backward evolution,
the susceptibility is nonincreasing at each fixed magnetization and that $\sqrt\chi$ flattens out for positive magnetizations. The inequality (S1) means that $\chi$
decreases at fixed $B$ as $s$ decreases, or equivalently that
$\chi+mB^2$ is concave in $B$. The inequality (S2) means that $\chi+mB^2$
is nonincreasing in $|B|$, and it follows from (S1) because an even
concave function is nonincreasing in $|B|$. The inequality (S3)
means that the slope $\partial_B\sqrt\chi\le0$ increases as $s$ decreases,
so that $\sqrt\chi$ flattens (\Cref{rem:slope}). Further, (S2) gives
$u_{xxx}=-2Cz\le0$ for $x>0$, which means that the magnetization
$u_x(s,\cdot)$ is concave on the positive half-line. This is an analog of the
Griffiths--Hurst--Sherman inequality for ferromagnetic Ising models
\cite{GHS1970}, which states that the susceptibility decreases as the
strength of a positive field increases. We prove (S1) by the maximum principle in backward time,
starting from the terminal profile. Its consequence (S2)
gives $z\ge0$, which makes the source term in the evolution
equation for $CH_x$ nonnegative and allows the same principle
to prove (S3). The explicit updates at the atoms preserve
both inequalities.

\medskip
\emph{The forward factor.} For a finitely supported measure, the
transformed density $\e^{-V(s,\cdot)}$ is a constant multiple of the
Gaussian density $g_{\beta^2s}$ before the first positive atom, evolves
by the heat equation between atoms, and is multiplied by
$\e^{-\mu(\{q\})u(q,\cdot)}$ at each positive atom $q$. At an atom of
mass $\delta$, the potential changes from $V$ to $V+\delta u$, which
preserves (S4) and (S5) because $u_{xx}>0$ and $u_{xxx}\le0$ for $x\ge0$.
Between atoms, $V_{xxx}$ satisfies a linear homogeneous parabolic
equation, and the maximum principle preserves
(S5). Writing
$V=-\log g_{\beta^2s}+W_s$, we find that $W_s''$ is nonincreasing on
$[0,\infty)$ while $W_s'$ is bounded, which forces $W_s''\ge0$ and hence
(S4) (\Cref{prop:forward}).

\medskip\noindent\emph{Step 4: Marginal stability.}
Once we know that $\supp\mu=[0,q_\beta]$, we have $\Gamma(s)=s$ for every
$s$ in the support, so that the stability condition holds with equality
$\Gamma'\equiv1$ along the whole support. This is the rigorous form of what the physics literature refers to as the marginal stability of the full
replica symmetry breaking phase. Differentiating the integrated form
\eqref{eq:Gamma-prime-integral} of $\Gamma'$, we find that
$\E[\chi^2\Phi]=0$ for almost every $s<q_\beta$, where now
$\Phi=\frac12\chi_B^2-\alpha(s)\chi$. Since this is affine in $\alpha(s)$,
we can solve for the distribution function. We show that the resulting
quotient is continuous, and since $\alpha$ is right continuous, we obtain
\[
 \alpha(s)=\frac{\E[\chi^2\chi_B^2]}{2\E\chi^3}
 =\frac{\E u_{xxx}(s,X_s)^2}{2\E u_{xx}(s,X_s)^3},
 \qquad 0\le s<q_\beta,
\]
with $\chi$ and $\chi_B$ evaluated at $(s,M_s)$ (see \Cref{lem:quotient}).
(This ratio has appeared previously in the proof of \cite[Theorem~2(ii)]{AuffingerChenProperties} and,
in the physics literature, \cite{SommersDupont1984} and
\cite[(22)--(23)]{CrisantiRizzo2002}.) Since the right side is
continuous and vanishes at $s=0$, the measure $\mu$ has no atoms in
$[0,q_\beta)$.
Further, we show that $\chi(q_\beta, B)=1-B^2$
and
\[
 \alpha(q_\beta-)
 =\frac{2\E[M^2\chi^2]}{\E[\chi^3]},
 \qquad M=\tanh X_{q_\beta},\quad
 \chi=\chi(q_\beta,M).
\]
This formula relates $\alpha(q_\beta-)$ to the magnetization
under the law reweighted by $\chi^2$. An integration by parts
at $q_\beta$, together with the forward curvature bound (S4),
then yields the quantitative atom bound in \Cref{thm:atom}.
Additionally, using an inductive argument, we show that the part of $\mu$ below
$q_\beta$ has a density that is smooth up to and including $q_\beta$
(\Cref{s:proof}). This sharpens
a regularity theorem of Auffinger and Chen \cite{AuffingerChenProperties}
(see \Cref{rem:AC-regularity}).

\medskip

The general strategy of using the self-consistency function
\(\Gamma\) and the optimality conditions at the endpoints of a
hypothetical support gap to prove interval support has a precursor in
spherical models, where the analog of $\Gamma$ is explicit and
concavity of $(\xi'')^{-1/2}$ can be used to verify full replica
symmetry breaking for suitable mixtures and temperatures 
\cite[Example 4]{AuffingerChenProperties}, based on
\cite[Proposition~2.1]{TalagrandSpherical}. For the SK model near criticality, the
strategy was carried out by Zhou \cite{ZhouFRSB}. The new ingredients here
are the backward inequalities (S1)--(S3), the
propagation of the forward shape properties through the atoms of $\mu$,
and the covariance identity \eqref{e:positive-identity-intro}. The forward
properties rest on a known fact about the heat flow, namely the
preservation of the sign of the third logarithmic derivative
\cite{EllisNewmanJDE1978}, which gives (S5). 
Together these ingredients apply to every gap and remove the
near-critical restriction. 

\subsection{Organization}
In \Cref{s:prelim}, we collect fundamental facts about the Parisi equation, the
optimal diffusion, and the optimality conditions that we use in the remainder of the paper. \Cref{s:backward} proves the backward inequalities (S1)--(S3)
(\Cref{prop:backward}), and \Cref{s:forward} proves the forward
inequalities (S4) and (S5) (\Cref{prop:forward}). The crossing property is then established in \Cref{s:crossing}
(\Cref{prop:crossing}), together with the shape of $\Gamma'$ on intervals
of constant mass (\Cref{cor:Gamma-prime-shape}). In \Cref{s:proof}, we
exclude gaps (\Cref{prop:connected-support}) and conclude the proofs of
the main theorems.

\subsection{Acknowledgments}
We thank Amol Aggarwal for helpful conversations.
This paper was prepared with the assistance of large language models,
which were used to suggest possible arguments, draft and revise the
manuscript, and write code for symbolic and numerical checks.
The author was partially supported by NSF grant DMS-2450004.

\section{Preliminaries}\label{s:prelim}

This section collects the tools used in the remainder of the paper: regularity of the Parisi
equation (\Cref{prop:PDE-regularity}), the optimal diffusion and the
function $\Gamma$, the explicit solution on intervals where $\alpha$ is
constant (\Cref{lem:constant-mass}), formulas for the derivatives of
$\Gamma$ (\Cref{lem:Gamma-derivatives}), the optimality conditions
(\Cref{lem:optimality}), and the comparison principle (\Cref{lem:comparison}). 

In subscripts, $x$ denotes a spatial derivative, as in
$u_x=\partial_xu$. For functions of one real variable, such as $\Gamma$,
$\alpha$, primes denote
derivatives. For a function $V(s,x)$ we write $V(s)=V(s,\cdot)$, and
$V(s)_x$, $V(s)_{xx}$, and so on, for its spatial derivatives.

\subsection{The Parisi equation}

For $t> 0$ we write
\[
 g_t(x)=\frac1{\sqrt{2\pi t}}\e^{-x^2/(2t)},
 \qquad
 P_tf(x)=\E f(x+\sqrt t\,Z),
\]
where $Z$ is a standard Gaussian random variable, so that $(P_t)$ is
the heat semigroup with generator $\frac12\partial_x^2$. For
$\nu\in\Pr([0,1])$ we write $\alpha_\nu(s)=\nu([0,s])$ and
$\alpha_\nu(s-)=\nu([0,s))$, with $\alpha_\nu(0-)=0$. We use
$\nu_n\Rightarrow\nu$ for weak convergence of probability measures.

For finitely supported $\nu$, the solution $u_\nu$ of
\eqref{eq:parisi-pde} is a classical solution on each interval between consecutive
atoms. For general $\nu$, $u_\nu$ is
the uniform limit of the solutions for finitely supported measures
converging weakly to $\nu$ \cite{Talagrand2006,AuffingerChenProperties}. The following proposition combines \cite[Propositions~1, 2, and~5]{AuffingerChenProperties}.

\begin{proposition}
\label{prop:PDE-regularity}
For every $\nu\in\Pr([0,1])$, the solution $u_\nu$ of
\eqref{eq:parisi-pde} is even in $x$, and all its positive-order spatial
derivatives are jointly continuous on $[0,1]\times\R$. For every $j\ge1$,
there exists a constant $K_j>0$ depending only on $\beta$ and $j$ such
that $|\partial_x^ju_\nu|\le K_j$ on $[0,1]\times\R$ for every
$\nu\in\Pr([0,1])$. Moreover,
\begin{equation}\label{eq:curvature-bounds}
 |u_{\nu,x}|\le1,
 \qquad 0<u_{\nu,xx}\le1.
\end{equation}
If
$\nu_n\Rightarrow\nu$, then for every $j\ge0$,
\begin{equation}\label{eq:uniform-derivative-convergence}
 \sup_{(s,x)\in[0,1]\times\R}
 |\partial_x^ju_{\nu_n}(s,x)-\partial_x^ju_\nu(s,x)|\longrightarrow0.
\end{equation}
\end{proposition}

By \Cref{prop:PDE-regularity}, $u_\nu(s,\cdot)$ is even and convex, so
$u_{\nu,x}(s,0)=0$ and
\begin{equation}\label{eq:u-linear-growth}
 0\le u_\nu(s,x)-u_\nu(s,0)\le|x|.
\end{equation}

\subsection{The optimal diffusion and approximation}

Let $(W_s)_{0\le s\le1}$ be a standard Brownian motion. For
$\nu\in\Pr([0,1])$, let $(X_s^\nu)_{0\le s\le1}$ be the unique strong
solution of
\begin{equation}\label{eq:optimal-diffusion}
 \dd X_s^\nu
 =\beta^2\alpha_\nu(s)u_{\nu,x}(s,X_s^\nu)\dd s+\beta\dd W_s,
 \qquad X_0^\nu=0.
\end{equation}
By \Cref{prop:PDE-regularity}, the drift is bounded by $\beta^2$ and
is $\beta^2$-Lipschitz in space, so the equation is well posed. This is
the optimal diffusion in the stochastic control representation of the
Parisi functional \cite{AuffingerChenUnique,JagannathTobasco2016,Chen2017}.
We set
\begin{equation}\label{eq:Gamma-def}
 \Gamma_\nu(s)=\E\bigl[u_{\nu,x}(s,X_s^\nu)^2\bigr].
\end{equation}
By $u_{\nu,x}(s,0)=0$, we have $\Gamma_\nu(0)=u_{\nu,x}(0,0)^2=0$.

\begin{lemma}\label{lem:approximation}
Let $\nu\in\Pr([0,1])$ and let $S\subset[0,1]$ be finite. There exist
finitely supported $\nu_n\Rightarrow\nu$ such that, for every $s\in S$,
\begin{equation}\label{eq:approximation-mass}
 \nu_n([0,s))=\nu([0,s)),\qquad \nu_n(\{s\})=\nu(\{s\}).
\end{equation}
Further, $\alpha_{\nu_n}\to\alpha_\nu$ in $L^1([0,1])$.
\end{lemma}

\begin{proof}
Adjoin $0,1$ to $S$ and partition each  open interval between points of $S$
into finitely many nonempty subintervals of diameter at most $1/n$.
Let $T_n$ fix $S$ and send each subinterval to one of its points, and set
$\nu_n=(T_n)_\#\nu$. Then $T_n$ has finite image and
$T_n^{-1}([0,s))=[0,s)$, $T_n^{-1}(\{s\})=\{s\}$ for $s\in S$,
proving \eqref{eq:approximation-mass}.
Since $\|T_n-\mathrm{id}\|_\infty\le1/n$, uniform continuity of
continuous test functions on $[0,1]$ gives $\nu_n\Rightarrow\nu$.
The distribution functions therefore converge almost everywhere,
and bounded convergence gives convergence in $L^1$.
\end{proof}

The following estimate lets us extend identities for finitely supported
measures to arbitrary measures along the optimal diffusion.

\begin{lemma}\label{lem:diffusion-stability}
Suppose that $\nu_n\Rightarrow\nu$, and drive $X^{\nu_n},X^\nu$ by the
same Brownian motion. Then, pathwise,
\begin{equation}\label{eq:diffusion-stability}
 \sup_{0\le s\le1}|X_s^{\nu_n}-X_s^\nu|
 \le\beta^2\e^{\beta^2}\left(
 \|u_{\nu_n,x}-u_{\nu,x}\|_\infty
 +\|\alpha_{\nu_n}-\alpha_\nu\|_{L^1}\right)\longrightarrow0.
\end{equation}
For a fixed polynomial $f$ in finitely many positive-order spatial
derivatives, let $f_\nu$ denote its evaluation at the derivatives of
$u_\nu$. Then
\begin{equation}\label{eq:moment-convergence}
 \sup_{0\le s\le1}
 \left|\E f_{\nu_n}(s,X_s^{\nu_n})-\E f_\nu(s,X_s^\nu)\right|
 \longrightarrow0.
\end{equation}
\end{lemma}

\begin{proof}
We abbreviate $u=u_\nu$ and $\alpha=\alpha_\nu$, write $u_n$ and
$\alpha_n$ for the corresponding objects for $\nu_n$, and set
$\Delta(s)=\sup_{0\le r\le s}|X_r^{\nu_n}-X_r^\nu|$. Subtracting the
integral equations for the coupled diffusions, we find that
\[
\begin{aligned}
 \Delta(s)
 &\le\beta^2\int_0^s\bigl|\alpha_n(r)u_{n,x}(r,X_r^{\nu_n})
 -\alpha(r)u_x(r,X_r^\nu)\bigr|\dd r\\
 &\le\beta^2\int_0^s\bigl(|\alpha_n(r)-\alpha(r)|
 +\|u_{n,x}-u_x\|_\infty+\Delta(r)\bigr)\dd r\\
 &\le\beta^2\bigl(\|\alpha_n-\alpha\|_{L^1}+\|u_{n,x}-u_x\|_\infty\bigr)
 +\beta^2\int_0^s\Delta(r)\dd r
\end{aligned}
\]
for all $s\in[0,1]$. In the second inequality, we used that $|u_{n,x}|\le1$
and $0\le\alpha\le1$, and that $u_x(r,\cdot)$ is $1$-Lipschitz because
$0<u_{xx}\le1$ by \eqref{eq:curvature-bounds}. By Gronwall's inequality, we
conclude that \eqref{eq:diffusion-stability} holds. Its right side tends to
zero by \eqref{eq:uniform-derivative-convergence} and because the
distribution functions converge at their continuity points, and hence in
$L^1$.

For the second assertion, let $j\ge1$. By \Cref{prop:PDE-regularity},
we have $|\partial_x^{j+1}u_{\nu'}|\le K_{j+1}$ for all
$\nu'\in\Pr([0,1])$. Therefore
\[
 \bigl|\partial_x^ju_n(s,X_s^{\nu_n})-\partial_x^ju(s,X_s^\nu)\bigr|
 \le\|\partial_x^ju_n-\partial_x^ju\|_\infty+K_{j+1}\Delta(1)
 \qquad\text{for all } s\in[0,1].
\]
By \eqref{eq:uniform-derivative-convergence} and
\eqref{eq:diffusion-stability}, the right side is bounded by a deterministic quantity and tends
to zero. Since a polynomial is Lipschitz on bounded sets, the same holds
for $f_{\nu_n}(s,X_s^{\nu_n})-f_\nu(s,X_s^\nu)$, and taking expectations
proves \eqref{eq:moment-convergence}.
\end{proof}

For finitely supported measures, we integrate the Parisi equation and
use $0\le u_{xx}+\alpha u_x^2\le2$ to obtain
\begin{equation}\label{eq:u-absolute-growth}
 0\le\log\cosh x\le u_\nu(s,x)
 \le\log\cosh x+\beta^2(1-s)\le\beta^2+|x|.
\end{equation}
By uniform convergence of $u_\nu$, these bounds extend to arbitrary
measures.

\subsection{Constant-mass intervals}

On an interval where $\alpha\equiv m$, the function $h=\e^{mu}$
solves a backward heat equation, with $h\equiv1$ when $m=0$.
The transition density of the optimal diffusion is the Gaussian
transition kernel multiplied by the ratio of $h$ at the final
and initial space-time points.

\begin{lemma}\label{lem:constant-mass}
Let $\nu\in\Pr([0,1])$, $0\le s_0<s_1\le1$, and suppose that
$\alpha_\nu(s)=m$ for $s_0\le s<s_1$. Write $u=u_\nu$ and $X=X^\nu$.
\begin{enumerate}
\item[(i)] If $m>0$, then
\begin{equation}\label{eq:cole-hopf}
 \e^{mu(t,\cdot)}
 =P_{\beta^2(s_1-t)}\e^{mu(s_1,\cdot)},
 \qquad s_0\le t\le s_1.
\end{equation}
If $m=0$, then $u(t,\cdot)=P_{\beta^2(s_1-t)}u(s_1,\cdot)$ for
$s_0\le t\le s_1$.
\item[(ii)] For $s_0<t\le s_1$, the law of $X_t$ has the density
\begin{equation}\label{eq:h-transform}
 p(t,x)=\e^{mu(t,x)}\int_\R g_{\beta^2(t-s_0)}(x-y)\,
 \e^{-mu(s_0,y)}\,\P(X_{s_0}\in\dd y).
\end{equation}
\end{enumerate}
\end{lemma}

\begin{proof}
First suppose that $\nu$ is finitely supported and $m>0$, and set
$h=\e^{mu}$. Then the Parisi PDE implies that
$h_\sigma+\beta^2h_{xx}/2=0$. For $s_0\le r<t\le s_1$, let
$Y_\sigma=y+\beta(W_\sigma-W_r)$. By It\^o's formula, we have
\[
 h(t,Y_t)-h(r,y)=\beta\int_r^t h_x(\sigma,Y_\sigma)\dd W_\sigma.
\]
Since $|h_x(\sigma,x)|\le m\e^{m(\beta^2+|x|)}$ by
\eqref{eq:u-absolute-growth} and Gaussian variables have exponential
moments, the stochastic integral is square-integrable. Taking expectations with $t=s_1$
proves (i). For $m=0$, we apply It\^o's formula directly to
$u(\sigma,Y_\sigma)$. Its diffusion coefficient $\beta u_x(\sigma,Y_\sigma)$
is bounded and its drift
vanishes, which gives the heat semigroup formula.

For arbitrary $\nu$, the approximants of \Cref{lem:approximation} preserving the masses at  $s_0,s_1$ preserve this interval of constant mass. Using the uniform convergence of $u_{\nu_n}$
and \eqref{eq:u-absolute-growth}, we pass (i) to the limit inside the
Gaussian integral. Differentiating under the Gaussian integral in (i),
we then find that $u$ is smooth in $(s,x)$ inside this interval and satisfies the Parisi equation.

For (ii), we fix $s_0<t\le s_1$ and $\varphi\in C_c^\infty(\R)$, and
we set $h=\e^{mu}$, including $h\equiv1$ when $m=0$. By (i) and the
semigroup property, $h(\sigma)=P_{\beta^2(t-\sigma)}h(t)$ for
$s_0\le\sigma\le t$. We define
\[
 F(\sigma,x)=\frac{P_{\beta^2(t-\sigma)}[h(t,\cdot)\varphi](x)}{h(\sigma,x)}.
\]
Since both the numerator and the denominator solve the backward heat
equation, the quotient rule and $h_x/h=mu_x$ give
\[
 F_\sigma+\frac{\beta^2}{2}F_{xx}+\beta^2mu_xF_x=0,
 \qquad F(t,\cdot)=\varphi.
\]
Since the Gaussian kernel is positive, $|F|\le\|\varphi\|_\infty$.
Differentiating the numerator under the convolution and using
$|h_x|\le mh$, we also find
\[
 |F_x|\le\|\varphi'\|_\infty+2m\|\varphi\|_\infty.
\]
By It\^o's formula along $X$, the drift vanishes and the stochastic
integrand is bounded, which yields
\[
 \E[\varphi(X_t)\mid\mathcal F_r]=F(r,X_r),
 \qquad s_0\le r<t.
\]
Specifically, we first apply It\^o's formula away from the endpoints, then pass
to them using continuity and bounded convergence for $F(\sigma,X_\sigma)$,
and It\^o's isometry for the stochastic integral. An atom at $s_1$ changes $\alpha_\nu$ only at the endpoint of
this interval, which has zero Lebesgue measure. Expanding the convolution in the definition of $F$, we
obtain
\[
 \E[\varphi(X_t)\mid\mathcal F_r]
 =\int_\R\varphi(x)\,\frac{g_{\beta^2(t-r)}(x-X_r)\,h(t,x)}{h(r,X_r)}\dd x.
\]
This holds for every $\varphi\in C_c^\infty(\R)$. Hence the conditional law
of $X_t$ given $\mathcal F_r$ has the density
$x\mapsto g_{\beta^2(t-r)}(x-X_r)h(t,x)/h(r,X_r)$, whose integral is one
because $h(r)=P_{\beta^2(t-r)}h(t)$ by (i). Taking $r=s_0$ and averaging over
$X_{s_0}$ proves (ii).
The same formula exhibits $p\e^{-mu}$ as a forward heat convolution,
which gives its smoothness in the interior and its heat equation.
\end{proof}

\subsection{Moment identities}

For $\nu\in\Pr([0,1])$ we write, along \eqref{eq:optimal-diffusion},
\begin{equation}\label{eq:M-C-D-A}
 M_s=u_{\nu,x}(s,X_s^\nu),\quad C_s=u_{\nu,xx}(s,X_s^\nu),\quad
 D_s=u_{\nu,xxx}(s,X_s^\nu),\quad A_s=u_{\nu,xxxx}(s,X_s^\nu).
\end{equation}
The first two moment identities below are stochastic-control versions
of the differentiation formulas in
\cite[Proposition~3]{AuffingerChenProperties}; in this form they appear
in \cite[Lemma~8.4]{JagannathTobasco2017}. We also record the closure of
polynomial moments under differentiation
(cf.\ \cite[Lemma~2]{AuffingerChenProperties}), which is used when we prove 
smoothness up to the terminal atom.

\begin{lemma}\label{lem:Gamma-derivatives}
Let $\nu\in\Pr([0,1])$ and write $\alpha=\alpha_\nu$, $u=u_\nu$, and
$\Gamma=\Gamma_\nu$. The process $M$ is a bounded martingale with
$\dd M_s=\beta C_s\dd W_s$. Further, $\Gamma\in C^1([0,1])$,
with
\begin{equation}\label{eq:Gamma-prime}
 \Gamma'(s)=\beta^2\E C_s^2,
\end{equation}
and
\begin{equation}\label{eq:Gamma-prime-integral}
 \Gamma'(t)-\Gamma'(r)=\beta^4\int_r^t
 \E[D_s^2-2\alpha(s)C_s^3]\dd s,
 \qquad 0\le r<t\le1.
\end{equation}
If $\alpha=m$ on an open interval $I$, then $\Gamma\in C^3(I)$ and
\begin{equation}\label{eq:Gamma-second-gap}
 \Gamma''(s)=\beta^4\E[D_s^2-2mC_s^3].
\end{equation}
More generally, let $f$ be a polynomial in finitely many positive-order
spatial derivatives of $u$, and let $e_f(s)=\E f(s,X_s^\nu)$. Then
$e_f$ is continuous on $[0,1]$, and there exist polynomials $f_0,f_1$
of the same kind such that
\begin{equation}\label{eq:moment-integral}
 e_f(t)-e_f(r)
 =\beta^2\int_r^t\bigl(e_{f_0}(s)+\alpha(s)e_{f_1}(s)\bigr)\dd s,
 \qquad 0\le r<t\le1.
\end{equation}
The polynomials $f_0,f_1$ depend only on the polynomial defining $f$.
\end{lemma}

\begin{proof}
First let $\nu$ be finitely supported, and set
$U_j(s)=\partial_x^ju(s,X_s)$ for $j\ge1$. By It\^o's formula
between atoms, we have
\begin{equation}\label{eq:MCD-SDE}
 \dd U_j
 =-\frac{\beta^2\alpha(s)}2
 \sum_{k=1}^{j-1}\binom jk U_{k+1}U_{j-k+1}\dd s
 +\beta U_{j+1}\dd W_s.
\end{equation}
The second-order term in It\^o's formula,
$\beta^2\partial_x^{j+2}u/2$, cancels the corresponding
negative term from the differentiated Parisi equation.
The drift of $X$ cancels the two terms containing $u_x$
in the expansion of $\partial_x^j(u_x^2)$. By continuity, the identities extend across atoms. The sum
is empty for $j=1$, giving $\dd M_s=\beta C_s\dd W_s$.
Let $f$ be a polynomial in $U_1,\dots,U_n$, and write $\partial_j$ for
the partial derivative in its $j$-th variable. By \eqref{eq:MCD-SDE} and
It\^o's formula, the process $f(U_1,\dots,U_n)$ has drift
$\beta^2(f_0+\alpha f_1)$ and diffusion coefficient
$\beta\sum_{j=1}^n\partial_jf\,U_{j+1}$, where
\begin{equation}\label{eq:f0-f1}
 f_0=\frac12\sum_{i,j=1}^n\partial_i\partial_jf\,U_{i+1}U_{j+1},
 \qquad
 f_1=-\frac12\sum_{j=1}^n\partial_jf\sum_{k=1}^{j-1}\binom jk
 U_{k+1}U_{j-k+1}.
\end{equation}
These are polynomials of the same kind, and they are bounded by
\Cref{prop:PDE-regularity}, as is the diffusion coefficient. Taking
expectations, we obtain \eqref{eq:moment-integral} for finitely supported
measures.

For arbitrary $\nu$, we pass this expectation identity to the limit
using \Cref{lem:diffusion-stability} and $\alpha_{\nu_n}\to\alpha$ in
$L^1$. Under the same coupling, $M$ and $C$ converge uniformly along the
approximants, and by the It\^o isometry the $j=1$ stochastic identity
passes to the limit, so $M$ is a bounded martingale. Every moment is continuous by
continuity of $X$ and bounded joint continuity of the spatial derivatives.

For $f=U_1^2$ and $f=U_2^2$, we find from \eqref{eq:f0-f1} that
$(f_0,f_1)=(C^2,0)$ and $(f_0,f_1)=(D^2,-2C^3)$, respectively. Taking
expectations, we obtain \eqref{eq:Gamma-prime} and
\eqref{eq:Gamma-prime-integral}.
On an interval where $\alpha=m$, \eqref{eq:Gamma-second-gap} follows by
continuity, and \eqref{eq:moment-integral}, applied to its two moments,
implies that $\Gamma\in C^3$ there.
\end{proof}

\subsection{Optimality conditions}
The following lemma restates  \cite[Theorem~5]{AuffingerChenProperties} and \cite[Proposition~1.1]{JagannathTobasco2017}. The first condition goes back to
\cite[Proposition~3.2]{TalagrandParisiMeasures}, and both conditions were
also obtained in \cite{Chen2017}. These references treat mixed models with
covariance $\E H_N(\sigma)H_N(\sigma')=N\xi(N^{-1}\sum_i\sigma_i\sigma_i')$;
for \eqref{eq:SK-Hamiltonian}, we have $\xi(q)=\beta^2q^2/2$ and
$\xi''=\beta^2$.
\begin{lemma}
\label{lem:optimality}
Let $\mu=\mu_\beta$. For every $q\in\supp\mu$,
\begin{equation}\label{eq:self-consistency}
 \Gamma_\mu(q)=q
 \qquad\text{and}\qquad
 \beta^2\E\bigl[u_{\mu,xx}(q,X_q^\mu)^2\bigr]\le1.
\end{equation}
\end{lemma}

\subsection{Comparison principle}

We use the following elementary comparison principle for bounded functions.
Only the outward component of the drift needs an upper bound. The
version for subsolutions follows by applying the statement to $-v$.

\begin{lemma}\label{lem:comparison}
Let $T>0$, $K\ge0$, and $I=\R$ or $I=[0,\infty)$. Let
$v\in C([0,T]\times I)\cap C^{1,2}((0,T]\times I^\circ)$ be bounded
and satisfy
\begin{equation}\label{eq:comparison-ineq}
 v_t-\frac12v_{xx}-bv_x-\kappa v\ge0
 \qquad\text{on }(0,T]\times I^\circ,
\end{equation}
where $b(t,x)\operatorname{sgn}(x)\le K$ and $\kappa\le K$.
If $v(0,\cdot)\ge0$ and, when $I=[0,\infty)$, $v(t,0)\ge0$, then
$v\ge0$ on $[0,T]\times I$.
\end{lemma}

\begin{proof}
Replacing $v$ by $\e^{-(K+1)t}v$, we may assume $\kappa\le-1$.
We set $\Theta(t,x)=\e^{(2K+2)t}(1+x^2)$. Since
$2bx\le2K|x|\le K(1+x^2)$,
\[
 \Theta_t-\tfrac12\Theta_{xx}-b\Theta_x-\kappa\Theta
 \ge\e^{(2K+2)t}\bigl((K+3)(1+x^2)-1\bigr)>0.
\]
For $\varepsilon>0$, $v+\varepsilon\Theta$ is a strict supersolution,
positive on the parabolic boundary and, because $v$ is bounded, for
sufficiently large $|x|$ uniformly in time. A nonpositive minimum would
therefore occur at positive time in the spatial interior. There its
time derivative is nonpositive, its first spatial derivative is zero,
and its second spatial derivative is nonnegative; since $\kappa\le-1$,
this contradicts the strict supersolution inequality. Let
$\varepsilon\rightarrow0$ to conclude.
\end{proof}
The use of the maximum principle to preserve the sign of the third
logarithmic derivative of a positive heat solution appears in
earlier work of Ellis and Newman \cite{EllisNewmanJDE1978}.
We apply this idea below to the forward potential, which is the
negative logarithm of a heat solution.

\section{Backward inequalities}\label{s:backward}

In this section, we prove the backward shape properties of Step~3 of the
outline, stated as \Cref{prop:backward}, and we explain in
\Cref{rem:slope} their interpretation in the magnetization coordinate. We then bound all derivative ratios needed below in \Cref{lem:relative-bounds}, compute the evolution of
$CQ$, $z$, and $CH_x$ in \Cref{lem:backward-equations}, and prove
\Cref{prop:backward} by induction over the intervals of constant mass.

Let $\nu\in\Pr([0,1])$ and $s\in[0,1]$ with
$m=\alpha_\nu(s)>0$. We write $u=u_\nu(s,\cdot)$, $B=u_x$, $C=u_{xx}$, and
define
\begin{equation}\label{eq:z-Q-H}
 z=-\frac{C_x}{2C},
 \qquad Q=z_x-mC,
 \qquad H=z^2+mC-2Q.
\end{equation}
These functions are well defined and smooth because $C>0$. The
functions $z$ and $H_x$ are odd, while $Q$ and $H$ are even.

\begin{proposition}\label{prop:backward}
Let $\nu\in\Pr([0,1])$ and $s\in[0,1]$ with $m=\alpha_\nu(s)>0$. With
the notation above, for every $x\ge0$ we have
\begin{equation}\label{eq:backward-bounds}
 mB\le z\le mB+1-m\le1,
 \qquad 0\le Q\le1-m,
 \qquad 0\le H_x\le6(1-m).
\end{equation}
In particular, $z$, $Q$, $H$, $H_x$, and $CH_x$ are bounded on $\R$, and
$u_{xxx}(s,x)\le0$ for every $x\ge0$.
\end{proposition}

\begin{remark}\label{rem:slope}
The meaning of \Cref{prop:backward} in terms of the magnetization $B=u_x$,
described in \Cref{s:outline}, rests on the following computation. On an
interval
$[q_0,q_1)$ where $\alpha_\nu=m$, we write the Parisi equation in the backward time
$\tau=\beta^2(q_1-s)$, so that $u_\tau=\frac12(u_{xx}+mu_x^2)$, and we let
$\chi(\tau,B)$ be the susceptibility $u_{xx}$ at the point $x(\tau,B)$
where $u_x=B$. Differentiating at fixed $B$, we find that
\[
 \partial_\tau\chi=\chi^2\Bigl(\frac12\chi_{BB}+m\Bigr),
 \qquad
 \partial_\tau x=z-mB,
 \qquad
 \partial_\tau\log\chi=-Q,
 \qquad
 \partial_B\partial_\tau\sqrt\chi=\frac{H_x}{4\sqrt\chi},
\]
together with $z=-\frac12\chi_B$ and $Q=-\chi(\frac12\chi_{BB}+m)$. Hence,
at fixed magnetization and along the backward evolution, the inequality
$z\ge mB$ for $B\ge0$ means that the position $x$ moves outward, $Q\ge0$
means that the susceptibility decreases, and $H_x\ge0$ means that the slope
$\partial_B\sqrt\chi\le0$ increases. In the forward time $t=\beta^2s$, the
first identity reads $\partial_t\chi=Q\chi$, which is \eqref{e:chipde}.
\end{remark}

We first prove \Cref{prop:backward} for finitely supported measures.
Until the end of its proof, we fix a finitely supported
$\nu\in\Pr([0,1])$ and $s\in[0,1]$ with $\alpha_\nu(s)>0$, and we work in
the backward time. For $0\le\tau\le T=\beta^2(1-s)$, we write
\[
 u(\tau,x)=u_\nu(1-\tau/\beta^2,x),
 \qquad a(\tau)=\alpha_\nu(1-\tau/\beta^2),
\]
so that $u(0,x)=\log\cosh x$ and $u_\tau=\frac12(u_{xx}+au_x^2)$ between
atoms. We set $u_j=\partial_x^ju$, $B=u_1$, and $C=u_2$. The atoms of $\nu$
in $(s,1]$ split $[0,T]$ into finitely many closed intervals such that $a$
is constant in the interior of each, and we call them constant-mass
intervals. If $\nu$ has an atom at $1$, the first of them is the single
point $\{0\}$ with $a=1$. On a constant-mass interval, we use the
constant value of $a$ also at the endpoints, and we define $z,Q,H$ by
\eqref{eq:z-Q-H} with $a$ in place of $m$. Thus $Q$ and $H$ may have different one-sided values at an
atom, while $z$, $u$, and all spatial derivatives of $u$
are continuous on $[0,T]\times\R$. Finally, on each
constant-mass interval we use the operator
\[
 \mathcal L=\frac12\partial_x^2+aB\partial_x,
 \qquad \mathcal T=\partial_\tau-\mathcal L,
\]
whose drift satisfies $|aB|\le1$ by \eqref{eq:curvature-bounds}.

Our first lemma bounds the size of $u_j$ for every $j$ such that $2 \le j \le 5$. 

\begin{lemma}\label{lem:relative-bounds}
There exists a constant $K_\beta>0$ depending only on $\beta$ such that
\begin{equation}\label{eq:relative-derivative-bounds}
 |u_j(\tau,x)|\le K_\beta C(\tau,x)
 \qquad\text{for all } \tau\in[0,T],\ x\in\R,\ \text{and } 2\le j\le5.
\end{equation}
In particular, the functions $z,z_x,z_{xx},Q,H,H_x$ are bounded on every
constant-mass interval by a constant depending only on $\beta$.
\end{lemma}

\begin{proof}
For $U=(C,u_3,u_4,u_5)$, differentiating the Parisi equation, we find
that on each constant-mass interval
\[
 \mathcal TU
 =a\bigl(C^2,\,3Cu_3,\,4Cu_4+3u_3^2,\,
                     5Cu_5+10u_3u_4\bigr)=MU,
\]
where $M$ is a lower triangular $4\times4$ matrix whose entries are
polynomials in $a$, $C$, and $u_3$. By \Cref{prop:PDE-regularity}, there
exists $K_\beta>0$ such that $|M\xi|\le K_\beta|\xi|$ for all
$\xi\in\R^4$, uniformly in $\tau$ and $x$. Since $C>0$, the Euclidean
norm $J=|U|$ is smooth and positive. The chain rule gives
\[
 \mathcal TJ=\frac{U\cdot MU}{J}
 -\frac{J^2|U_x|^2-(U\cdot U_x)^2}{2J^3}\le K_\beta J,
\]
where we used the Cauchy--Schwarz inequality twice, once in the form
$U\cdot MU\le K_\beta J^2$, and once to see that the second term is
nonpositive.

At $\tau=0$ we have $u=\log\cosh x$ and $B=\tanh x$, so that
$u_3=-2BC$, $u_4=2(3B^2-1)C$, and $u_5=-8B(3B^2-2)C$. Therefore,
\[
 \frac{J(0,x)^2}{C(0,x)^2}=576B^6-732B^4+236B^2+5\le85
 \qquad\text{for all } x\in\R,
\]
since $85-(576B^6-732B^4+236B^2+5)=(1-B^2)(576B^4-156B^2+80)$ and the
last factor is positive. In particular, $J(0,\cdot)\le10C(0,\cdot)$.

We now set $v=10\e^{K_\beta\tau}C-J$. Since $\mathcal TC=aC^2\ge0$, we
have
\[
 \mathcal Tv-K_\beta v
 =10\e^{K_\beta\tau}\mathcal TC-(\mathcal TJ-K_\beta J)\ge0
\]
on each constant-mass interval. The function $v$ is bounded and nonnegative at
$\tau=0$, and it is continuous on $[0,T]\times\R$ because $U$ and $C$ are
continuous across atoms. Applying \Cref{lem:comparison} with $b=aB$ and
$\kappa=K_\beta$ on each constant-mass interval in turn, we find that
$v\ge0$ on $[0,T]\times\R$. Since $|u_j|\le J$ and $T\le\beta^2$, this
proves \eqref{eq:relative-derivative-bounds} after replacing $K_\beta$ by
$10\e^{\beta^2K_\beta}$.

Finally, each of the functions $z,z_x,z_{xx},Q,H,H_x$ is a polynomial in
$a$, $C$, and the ratios $u_j/C$ with $2\le j\le5$. For example, $z=-u_3/(2C)$ and $z_x=-u_4/(2C)+2z^2$. Together with
\eqref{eq:relative-derivative-bounds} and $0\le a\le1$, this proves the
last assertion.
\end{proof}

The second lemma records the evolution equations for the three quantities that we will 
compare.

\begin{lemma}\label{lem:backward-equations}
On each constant-mass interval, we have
\begin{equation}\label{eq:weighted-backward-equations}
\begin{aligned}
 \mathcal T(CQ)&=-2Q(CQ),\\
 \mathcal Tz&=-2Qz,\\
 \mathcal T(CH_x)&=-5Q(CH_x)+6CzQ^2.
\end{aligned}
\end{equation}
If $a=1$ at $\tau=0$, then $z=B$ and $Q=H_x=0$ there. If two consecutive constant-mass intervals meet at an atom of mass
$\delta$, crossing the atom in backward time (decreasing $s$)
changes $a$ to $a-\delta$ and gives
\begin{equation}\label{eq:backward-jumps}
 CQ\longmapsto CQ+\delta C^2,
 \qquad z\longmapsto z,
 \qquad CH_x\longmapsto CH_x+6\delta C^2z.
\end{equation}
\end{lemma}

\begin{proof}
The operator $\mathcal T$ satisfies the product and chain rules
\[
 \mathcal T(fg)=f\,\mathcal Tg+g\,\mathcal Tf-f_xg_x,
 \qquad
 \mathcal T(\log f)=\frac{\mathcal Tf}f+\frac12(\log f)_x^2,
\]
and the commutation rule $\mathcal T(f_x)=(\mathcal Tf)_x+aCf_x$. Since
$\mathcal TC=aC^2$ and $(\log C)_x=-2z$, we have
$\mathcal T(\log C)=aC+2z^2$. The commutation rule then gives
$\mathcal T((\log C)_x)=4z(z_x-aC)=4zQ$, which is the equation for $z$.
Applying the commutation rule to $z$ and using $\mathcal TC=aC^2$ once
more, we find that $Q=z_x-aC$ satisfies
$\mathcal TQ=-2zQ_x-Q(2Q+aC)$. The product rule with $C_x=-2Cz$ then
gives the equation for $CQ$. For the last identity, we write
$P=CQ$ and use $H_x=2zQ-2Q_x$, which gives $CH_x=-2(P_x+zP)$. By the
commutation and product rules and the first two equations, we have
\[
 \mathcal T(P_x)=-2Q_xP-2QP_x+aCP_x,
 \qquad
 \mathcal T(zP)=-4QzP-z_xP_x.
\]
Since $z_x=Q+aC$, it follows that
\[
 \mathcal T(CH_x)=4Q_xP+8QzP+(4Q-2aC+2z_x)P_x
 =4CQQ_x+8CzQ^2+6QP_x.
\]
Substituting $Q_x=zQ-\frac12H_x$ and $P_x=-\frac12CH_x-zP$, we find the
third equation.

At $\tau=0$, we have $u_3=-2BC$, so that $z=B$ and $z_x=C$. If $a=1$, it
follows that $Q=0$ and $H=B^2+C=1$, and therefore $H_x=0$. At an atom of
mass $\delta$, the parameter $a$ decreases by $\delta$, while $u$ and
all its spatial derivatives are continuous. Since $Q=z_x-aC$ and
$H=z^2+aC-2Q$, the function $Q$ increases by $\delta C$ and $H$
decreases by $3\delta C$. Since $C_x=-2Cz$, this proves
\eqref{eq:backward-jumps}.
\end{proof}

\begin{proof}[Proof of \Cref{prop:backward}]
We first treat finitely supported $\nu$, with the notation introduced
above. We consider the functions
\[
 v_1=CQ,\qquad v_2=(1-a)C-CQ,\qquad v_3=aB+1-a-z,\qquad
 v_4=CH_x,\qquad v_5=6(1-a)C-CH_x,
\]
and we claim that
\begin{equation}\label{eq:backward-claim}
 v_1,v_2\ge0\ \text{on }\R,
 \qquad v_3,v_4,v_5\ge0\ \text{on }[0,\infty)
\end{equation}
for all $\tau\in[0,T]$, with both one-sided values at each atom. We first note that
\eqref{eq:backward-claim} at $\tau=T$ implies \eqref{eq:backward-bounds}
because $a(T)=m$. Indeed, since $C>0$, the inequalities for $v_1$ and
$v_2$ mean that $0\le Q\le1-m$. For $x\ge0$, we integrate $(z-mB)_x=Q\ge0$
from $z=B=0$ at $x=0$ to find that $z\ge mB$, and $v_3\ge0$ means that
$z\le mB+1-m\le1$. Finally, the inequalities for $v_4$ and $v_5$ mean that
$0\le H_x\le6(1-m)$.

We prove \eqref{eq:backward-claim} by induction over the
constant-mass intervals in order of increasing backward time
$\tau=\beta^2(1-s)$. For each interval, the induction hypothesis
is that \eqref{eq:backward-claim} holds at its initial endpoint
in $\tau$, corresponding to its right endpoint in $s$. The first interval contains $\tau=0$ and has $a=1$. There \Cref{lem:backward-equations} gives $z=B$ and $Q=H_x=0$,
so that $v_1,\dots,v_5$ all vanish at $\tau=0$.

Suppose now that \eqref{eq:backward-claim} holds at the left endpoint in $\tau$ of a
constant-mass interval $I$. By \eqref{eq:weighted-backward-equations},
$\mathcal TB=0$, and $\mathcal TC=aC^2$, we have on $I$
\begin{equation}\label{eq:backward-barrier-equations}
\begin{aligned}
 \mathcal Tv_1+2Qv_1&=0, &
 \mathcal Tv_2+2Qv_2&=(1-a)C(aC+2Q),\\
 \mathcal Tv_3+2Qv_3&=2Q(aB+1-a), &
 \mathcal Tv_4+5Qv_4&=6CzQ^2,
\end{aligned}
\end{equation}
and
\[
 \mathcal Tv_5+5Qv_5=6C\bigl[(1-a)(aC+5Q)-zQ^2\bigr].
\]
By \Cref{lem:relative-bounds} and $0<C\le1$, the functions
$v_1,\dots,v_5$ and the coefficients $2Q$ and $5Q$ are bounded on $I$, so
we may apply \Cref{lem:comparison} on $I$ with $b=aB$ and $\kappa=-2Q$ or
$\kappa=-5Q$. We use the five equations in turn. First, we find from the equation for
$v_1$ that $v_1\ge0$ on $I\times\R$, and therefore $Q\ge0$ there. Then
the right side of the equation for $v_2$ is nonnegative, so that
$v_2\ge0$ on $I\times\R$ as well. Since $Q\ge0$ and $B\ge0$ on
$[0,\infty)$, the right side of the equation for $v_3$ is nonnegative on
$I\times[0,\infty)$. Integrating $(z-aB)_x=Q\ge0$ from $z=B=0$ at $x=0$,
we find that $z\ge aB\ge0$ for $x\ge0$, so that the right side of the
equation for $v_4$ is nonnegative there as well. Since $z,B,H_x$ are
odd, we have $v_3=1-a\ge0$ and $v_4=0$ at $x=0$, and half-line comparison
implies that $v_3,v_4\ge0$ on $I\times[0,\infty)$. In particular, $0\le z\le aB+1-a\le1$ there.
Together with $0\le Q\le1-a$, this implies that $zQ^2\le(1-a)Q$, so that
the right side of the equation for $v_5$ is nonnegative. Since
$v_5=6(1-a)C\ge0$ at $x=0$, we conclude by half-line comparison that
$v_5\ge0$ on $I\times[0,\infty)$.

It remains to check that the inequalities persist across an atom of mass
$\delta$ at the right endpoint  of $I$ in $\tau$. There $a$ decreases by $\delta$, and
\eqref{eq:backward-jumps} implies that $v_1,\dots,v_5$ increase by
\[
 \delta C^2,\qquad \delta C(1-C),\qquad \delta(1-B),\qquad
 6\delta C^2z,\qquad 6\delta C(1-Cz),
\]
respectively. Since $0<C\le1$ and $|B|\le1$, and since $0\le z\le1$ on
$[0,\infty)$, these increments are nonnegative on $\R$ for the first two
and on $[0,\infty)$ for the others. Hence \eqref{eq:backward-claim} holds at the left
endpoint of the next interval, which completes the induction.

For general $\nu$, we choose finitely supported $\nu_n$ as in
\Cref{lem:approximation}, with $\alpha_{\nu_n}(s)=m$. Since the spatial
derivatives converge uniformly and $C>0$, the functions $z,Q,H,H_x$
converge pointwise, and \eqref{eq:backward-bounds} passes to the limit.
By parity and the definition of $H$, the functions $z,Q,H,H_x$ are bounded
on $\R$. So is $CH_x$ because $C\le1$.
Finally, $u_{xxx}=C_x=-2Cz\le0$ for $x\ge0$, which is the last assertion.
\end{proof}

\section{Forward inequalities}\label{s:forward}

We now turn to the forward factor of Step~3 of the outline and prove
\Cref{prop:forward}. For finitely supported measures, we write each heat
step as a Gaussian expectation of the correction in \Cref{lem:heat-step}
and bound all its derivatives in \Cref{lem:forward-derivative-bounds}. In
the proof of \Cref{prop:forward}, we derive the slope and curvature bounds,
propagate the sign of the third derivative by comparison, and pass to
arbitrary measures.

\begin{proposition}\label{prop:forward}
Let $\nu\in\Pr([0,1])$, $s\in(0,1]$, and
$\lambda_s=1/(\beta^2s)$. The law of $X_s^\nu$ has a smooth, strictly
positive density $p(s,\cdot)$. Define the transformed potentials
\begin{equation}\label{eq:V-def}
 V(s,x)=\alpha_\nu(s)u_\nu(s,x)-\log p(s,x),
 \qquad
 V(s-,x)=\alpha_\nu(s-)u_\nu(s,x)-\log p(s,x),
\end{equation}
and write $V(s)=-\log g_{\beta^2s}+W_s$ and
$V(s-)=-\log g_{\beta^2s}+W_{s-}$. Thus
\begin{equation}\label{eq:forward-density}
 p(s,x)=g_{\beta^2s}(x)\e^{\alpha_\nu(s)u_\nu(s,x)-W_s(x)},
 \qquad W_s=W_{s-}+\nu(\{s\})u_\nu(s,\cdot).
\end{equation}
The potentials and correction terms are smooth and even in $x$. Every positive-order
spatial derivative of $W_s$ and $W_{s-}$ is bounded by a constant depending
only on $\beta$ and its order, uniformly in $\nu,s,x$. Moreover,
$|W_s'|,|W_{s-}'|\le1$. Further:
\begin{enumerate}
\item[(a)] On $\R$,
\[
 \lambda_s\le V(s)_{xx}\le\lambda_s+\alpha_\nu(s),
\]
while $V(s)_{xxx}\le0$ on $[0,\infty)$.
The same holds for $V(s-)$ with $\alpha_\nu(s-)$ in place of
$\alpha_\nu(s)$.
\item[(b)] The density satisfies
\begin{equation}\label{eq:forward-density-tail}
\begin{aligned}
 p(s,x)&\le\e^{\beta^2+|x|}g_{\beta^2s}(x),\\
 p(s,x)&\le p(s,0)
 \exp\bigl(\alpha_\nu(s)|x|-\lambda_sx^2/2\bigr).
\end{aligned}
\end{equation}
\end{enumerate}
If $\nu_n$ are the approximants of \Cref{lem:approximation} at $s$, then
\begin{equation}\label{eq:forward-convergence}
 W_s^{\nu_n}\longrightarrow W_s^\nu,
 \qquad W_{s-}^{\nu_n}\longrightarrow W_{s-}^\nu
 \qquad\text{in }C^\infty_{\mathrm{loc}}(\R).
\end{equation}
\end{proposition}

We first prove \Cref{prop:forward} for finitely supported measures.
Until the end of Step~2 of its proof, we fix a finitely supported
$\nu\in\Pr([0,1])$ and use the forward time $t=\beta^2s$. We set
\[
 w(t,x)=W_{t/\beta^2}(x),\qquad \widetilde V(t,x)=-\log g_t(x)+w(t,x),
\]
so that $\e^{-\widetilde V}=g_t\e^{-w}$ is the transformed density
$\e^{-V(t/\beta^2,\cdot)}$. Applying \Cref{lem:constant-mass}(ii) on
successive intervals, we see that the law of $X_s^\nu$ has a smooth, strictly
positive density for every $s\in(0,1]$. The correction $w$ changes in two
ways. Before the first positive atom on $[0,1]$, it is the constant
$\alpha_\nu(0)u_\nu(0,0)$. At an atom of mass $\delta$ at time $s$, it
increases by $\delta u_\nu(s,\cdot)$, by \eqref{eq:forward-density}.
Between atoms, again by \Cref{lem:constant-mass}(ii), the function
$\e^{-\widetilde V}$ solves the heat equation
$\partial_t\e^{-\widetilde V}=\frac12\partial_x^2\e^{-\widetilde V}$. The first lemma expresses this heat step in
terms of $w$.

\begin{lemma}\label{lem:heat-step}
Let $0<r<t$, and suppose that $\nu$ is finitely supported with no atom in
$(r/\beta^2,t/\beta^2)$. We take the value of $w(r,\cdot)$ after a possible
atom at $r/\beta^2$ and the value of $w(t,\cdot)$ before a possible atom
at $t/\beta^2$, and let $Z$ be a standard Gaussian variable. Then
\begin{equation}\label{eq:heat-step}
 w(t,x)=-\log\E\exp\left[-w\left(r,
       \frac rt x+\sqrt{\frac{r(t-r)}t}\,Z\right)\right]
 \qquad\text{for all } x\in\R.
\end{equation}
Further, $w(t,\cdot)$ is even for every $t>0$.
\end{lemma}

\begin{proof}
By \eqref{eq:h-transform} with $s_0=r/\beta^2$, the function
$\e^{-\widetilde V(t,\cdot)}$ is the convolution of
$\e^{-\widetilde V(r,\cdot)}$ with $g_{t-r}$,
that is,
\[
 g_t(x)\e^{-w(t,x)}=\int_\R g_{t-r}(x-y)\,g_r(y)\e^{-w(r,y)}\dd y.
\]
Completing the square in the exponent, we find that
\[
 g_{t-r}(x-y)\,g_r(y)=g_t(x)\,g_{r(t-r)/t}\bigl(y-\tfrac rtx\bigr).
\]
We insert this identity into the integral and divide by $g_t(x)$. The
integral then becomes the expectation of $\e^{-w(r,Y)}$ for a Gaussian
variable $Y$ with mean $rx/t$ and variance $r(t-r)/t$, which is
\eqref{eq:heat-step}. The function $w$ is constant before the first
positive atom. Both \eqref{eq:heat-step} and the increments
$\delta u_\nu(s,\cdot)$ preserve evenness because $Z$ is symmetric and
$u_\nu(s,\cdot)$ is even. Since $\nu$ has finitely many atoms, $w(t,\cdot)$
is even for every $t>0$.
\end{proof}

The second lemma bounds every spatial derivative of the correction
term $w$ uniformly in $\nu$.

\begin{lemma}\label{lem:forward-derivative-bounds}
For every $k\ge1$, there exists a constant $c_k>0$ depending only on
$\beta$ and $k$ such that
\[
 |\partial_x^jw(t,x)|\le c_k
 \qquad\text{for all } t\in(0,\beta^2],\ x\in\R,\ \text{and } 1\le j\le k.
\]
This holds for both one-sided values at the atoms of $\nu$.
\end{lemma}

\begin{proof}
We fix $k\ge1$ and set $F=\e^{-w}$. Let $0<t_1<\dots<t_n\le\beta^2$ be the
forward times of the positive atoms of $\nu$, let
$\delta_1,\dots,\delta_n$ be their masses, and set $t_0=0$ and
$t_{n+1}=\beta^2$. We claim that, for $0\le i\le n$, the bounds
\begin{equation}\label{eq:F-induction}
 0<F\le1
 \qquad\text{and}\qquad
 |F^{(j)}|\le L_iF\quad\text{for all } 0\le j\le k
\end{equation}
hold for $t_i<t<t_{i+1}$ and, if $t_i<t_{i+1}$, at
$t_{i+1}$ using the value before any atom there.
For $i\ge1$, they also hold at $t_i$ using the value after
the atom there, including the case
$t_n=t_{n+1}=\beta^2$. Here $L_0=1$ and
$L_i=(1+c\delta_i)L_{i-1}$ for a constant $c$ depending only on $\beta$
and $k$. We argue by induction on $i$.

For $i=0$, the function $F$ is the constant
$\e^{-\alpha_\nu(0)u_\nu(0,0)}\in(0,1]$, since $u_\nu\ge0$ by
\eqref{eq:u-absolute-growth}. Suppose that the claim holds for some
$i<n$. At the atom at $t_{i+1}$, the function $F$ is multiplied by
$\e^{-\delta u}$, where $\delta=\delta_{i+1}$ and
$u=u_\nu(t_{i+1}/\beta^2,\cdot)$. For $1\le l\le k$, the derivative
$\partial_x^l\e^{-\delta u}$ is $\e^{-\delta u}$ times a polynomial in
$\delta u_x,\dots,\delta\partial_x^lu$ without constant term. Since
$\delta\le1$ and the bounds in \Cref{prop:PDE-regularity} depend only on
$\beta$, it follows that
$|\partial_x^l\e^{-\delta u}|\le c'\delta\e^{-\delta u}$ for a constant
$c'$ depending only on $\beta$ and $k$. By the product rule, we have
\[
 \bigl|\partial_x^j(F\e^{-\delta u})\bigr|
 \le\sum_{l=0}^j\binom jl|F^{(j-l)}|\,|\partial_x^l\e^{-\delta u}|
 \le(1+2^kc'\delta)L_iF\e^{-\delta u}
\]
for $0\le j\le k$. Since $u\ge0$, we also have $0<F\e^{-\delta u}\le1$.
With $c=2^kc'$, this proves \eqref{eq:F-induction} with the constant
$L_{i+1}$ after the atom at $t_{i+1}$. For $t\in(t_{i+1},t_{i+2}]$,
\Cref{lem:heat-step} with $r=t_{i+1}$ gives
$F(t,x)=\E F(t_{i+1},\theta x+bZ)$, where $\theta=t_{i+1}/t\in(0,1)$ and
$b>0$. Since the derivatives of
$F(t_{i+1},\cdot)$ up to order $k$ are bounded, we may differentiate under
the expectation, and we find that
\[
 |\partial_x^jF(t,x)|
 \le\theta^j\E|\partial_x^jF(t_{i+1},\theta x+bZ)|
 \le L_{i+1}\E F(t_{i+1},\theta x+bZ)=L_{i+1}F(t,x)
\]
for all $x\in\R$ and $0\le j\le k$. In addition, $0<F(t,\cdot)\le1$. This
proves the claim for $i+1$ and completes the induction.

Since $\sum_i\delta_i\le1$, we obtain
$L_i\le\prod_i(1+c\delta_i)\le\e^c$ for every $i$. Each derivative of
$w=-\log F$ of order $j\le k$ is a polynomial in $F'/F,\dots,F^{(j)}/F$. For
example, $w_x=-F'/F$ and $w_{xx}=-F''/F+(F'/F)^2$. This proves the lemma.
\end{proof}

\begin{proof}[Proof of \Cref{prop:forward}]
\emph{Step 1: Slope and upper curvature bound.} Let $\nu$ be finitely
supported, and let $r<t$ be as in \Cref{lem:heat-step}. We write $\theta=r/t$,
$Y=\theta x+\sqrt{r(t-r)/t}\,Z$, and
\[
 \langle f\rangle=\frac{\E[f(Y)\e^{-w(r,Y)}]}{\E\e^{-w(r,Y)}}
\]
for the tilted expectation in \eqref{eq:heat-step}. By
\Cref{lem:forward-derivative-bounds}, we may differentiate
\eqref{eq:heat-step} under the expectation. Differentiating the logarithm
twice, we find
\[
 w_x(t,x)=\theta\langle w_x(r,\cdot)\rangle,
 \qquad
 w_{xx}(t,x)=\theta^2\bigl(\langle w_{xx}(r,\cdot)\rangle
                     -\Var_{\langle\cdot\rangle}(w_x(r,\cdot))\bigr).
\]
Since $0<\theta<1$ and the variance is nonnegative, we deduce that
\begin{equation}\label{eq:heat-step-monotone}
 |w_x(t,x)|\le\sup_{y\in\R}|w_x(r,y)|,
 \qquad
 w_{xx}(t,x)\le\max\Bigl\{0,\sup_{y\in\R}w_{xx}(r,y)\Bigr\}
\end{equation}
for all $x\in\R$. At an atom of mass $\delta$, the functions $w_x$ and
$w_{xx}$ change by $\delta u_{\nu,x}$ and $\delta u_{\nu,xx}$, and
$|u_{\nu,x}|\le1$ and $u_{\nu,xx}\le1$ by \eqref{eq:curvature-bounds}.
We conclude that the two right sides in \eqref{eq:heat-step-monotone}
do not increase between atoms and increase by at most $\delta$ at an atom
of mass $\delta$. Since they vanish before the first positive atom, where
$w$ is constant, and since the masses of the atoms in $(0,s]$ sum to at
most $\alpha_\nu(s)$, we obtain
\begin{equation}\label{eq:forward-slope-curvature}
 |w_x(t,x)|\le\alpha_\nu(s)\le1
 \qquad\text{and}\qquad
 w_{xx}(t,x)\le\alpha_\nu(s)
\end{equation}
for all $t=\beta^2s\in(0,\beta^2]$ and $x\in\R$. The same bounds hold with
$\alpha_\nu(s-)$ in place of $\alpha_\nu(s)$ for the values before an atom
at $s$.

\emph{Step 2: Third derivative and convexity.}
Between atoms, inserting $\e^{-\widetilde V}=g_t\e^{-w}$ into the heat equation, we
find that
\[
 w_t=\frac12w_{xx}-\frac xt w_x-\frac12w_x^2.
\]
Differentiating three times in $x$, we find that
$R=\widetilde V_{xxx}=w_{xxx}$ satisfies
\begin{equation}\label{eq:forward-shape-evolution}
 R_t=\frac12R_{xx}-\widetilde V_xR_x-3\widetilde V_{xx}R.
\end{equation}
We claim that $R\le0$ on $(0,\beta^2]\times[0,\infty)$, for both one-sided
values at atoms. Before the first positive atom, $w$ is constant and $R=0$.
At an atom of mass $\delta$ at $s$, the function $R$ changes by
$\delta u_{\nu,xxx}(s,\cdot)$, which is nonpositive on $[0,\infty)$ by
\Cref{prop:backward}. Let $0<t_1<\dots<t_n\le\beta^2$ be the forward times of
the positive atoms, set $t_{n+1}=\beta^2$, and suppose that
$R(t_i,\cdot)\le0$ on $[0,\infty)$ after the atom at $t_i$ for some
$1\le i\le n$.
If $t_i=t_{i+1}$, the required bound at this final time
already holds by the induction hypothesis. We therefore
assume $t_i<t_{i+1}$. The coefficients in 
\eqref{eq:forward-shape-evolution} are smooth up to the endpoints,
using the values after the atom at $t_i$ and before the atom at
$t_{i+1}$. By \eqref{eq:forward-slope-curvature} and
\Cref{lem:forward-derivative-bounds}, we have
\[
 -\widetilde V_x=-\frac xt-w_x\le1
 \qquad\text{and}\qquad
 -3\widetilde V_{xx}=-\frac3t-3w_{xx}\le3c_2
 \qquad\text{for } x\ge0.
\]
Since $R$ is bounded by \Cref{lem:forward-derivative-bounds} and vanishes
at $x=0$ because $w(t,\cdot)$ is even, we may apply \Cref{lem:comparison}
on the half-line to $-R$ with $b=-\widetilde V_x$ and
$\kappa=-3\widetilde V_{xx}$, and we
conclude that $R\le0$ on $[t_i,t_{i+1}]\times[0,\infty)$. By induction over the
atoms, this proves the claim.

Since $w_{xxx}=R\le0$ on $[0,\infty)$, the function $w_{xx}(t,\cdot)$ is
nonincreasing there. It cannot be negative at any $x_0\ge0$, since then
$w_x(t,x)\le w_x(t,x_0)+(x-x_0)w_{xx}(t,x_0)$ for $x\ge x_0$, and the
bounded function $w_x(t,\cdot)$ would tend to $-\infty$. By evenness and
\eqref{eq:forward-slope-curvature}, we conclude that
$0\le w_{xx}\le\alpha_\nu(s)$ on $\R$. Since $V(s)_{xx}=\lambda_s+w_{xx}$
and $V(s)_{xxx}=R$, this proves (a) for finitely supported measures, and
the same argument applies to $V(s-)$.

At the origin, we have $w_x(t,0)=0$, and the equation for $w_t$ gives
$w_t(t,0)=\frac12w_{xx}(t,0)\in[0,\frac12]$. The initial value
$\alpha_\nu(0)u_\nu(0,0)$ and the increments $\delta u_\nu(s,0)$ at the
atoms are nonnegative, and their sum is at most $\beta^2$ by
\eqref{eq:u-absolute-growth}. Integrating in $t$, we obtain
\[
 0\le w(t,0)\le\beta^2+\frac t2
 \qquad\text{for all } t\in(0,\beta^2].
\]
Since $w(t,\cdot)$ is even and convex, we have $w\ge w(t,0)\ge0$. Together
with \eqref{eq:forward-density} and \eqref{eq:u-absolute-growth}, this
proves the first bound in (b).

\emph{Step 3: Arbitrary measures.}
Let $\nu\in\Pr([0,1])$ and $s\in(0,1]$. We choose $\nu_n$ as in
\Cref{lem:approximation} at $s$, so that $\alpha_{\nu_n}(s)=\alpha_\nu(s)$
and $\nu_n(\{s\})=\nu(\{s\})$. By \eqref{eq:diffusion-stability}, we have
$X_s^{\nu_n}\to X_s^\nu$ almost surely, and hence in law. By
\Cref{lem:forward-derivative-bounds} and the bound at the origin, the
functions $W_s^{\nu_n}$ and all their derivatives are bounded on every
compact set uniformly in $n$. By the Arzel\`a--Ascoli theorem and a
diagonal argument, every subsequence has a further subsequence along which
$W_s^{\nu_n}$ converges in $C^\infty_{\mathrm{loc}}(\R)$ to some limit $W$.
Along it, \eqref{eq:forward-density} and
\eqref{eq:uniform-derivative-convergence} imply that the densities
$p_n(s,\cdot)$ of $X_s^{\nu_n}$ converge locally uniformly to
$\widetilde p=g_{\beta^2s}\e^{\alpha_\nu(s)u_\nu(s,\cdot)-W}$. For every
$\varphi\in C_c(\R)$, it follows that
\[
 \E\varphi(X_s^\nu)=\lim_{n\to\infty}\E\varphi(X_s^{\nu_n})
 =\lim_{n\to\infty}\int_\R\varphi\,p_n(s,\cdot)\dd x
 =\int_\R\varphi\,\widetilde p\dd x,
\]
where the limits are taken along the subsequence. We conclude that
$\widetilde p$ is the density $p(s,\cdot)$ of $X_s^\nu$, which is therefore
smooth and strictly positive. In particular, the limit
$W=\alpha_\nu(s)u_\nu(s,\cdot)-\log p(s,\cdot)+\log g_{\beta^2s}$ does not
depend on the subsequence, and the whole sequence converges. This is the
first convergence in \eqref{eq:forward-convergence}. The second follows
from $W_{s-}^{\nu_n}=W_s^{\nu_n}-\nu(\{s\})u_{\nu_n}(s,\cdot)$ and
$\nu_n(\{s\})=\nu(\{s\})$. The derivative bounds, (a), the first bound in
(b), and evenness pass to both limits.

Finally, for arbitrary $\nu$,
\[
 \frac{p(s,x)}{p(s,0)}
 =\exp\bigl(\alpha_\nu(s)(u_\nu(s,x)-u_\nu(s,0))
 -(V(s,x)-V(s,0))\bigr).
\]
By evenness and (a), we have $V(s,x)-V(s,0)\ge\lambda_sx^2/2$. Together
with \eqref{eq:u-linear-growth}, this proves the second bound in (b).
\end{proof}

\section{The crossing property}\label{s:crossing}

In this section, we prove the crossing property, \Cref{prop:crossing}, by
implementing Step~2 of \Cref{s:outline} using the shape properties of Step~3. In \Cref{lem:transport}, we compute the evolution of the weight
$\omega=2pC^2$ under the velocity $v=mu_x-z$, which keeps the
magnetization $B=u_x$ fixed, meaning
$(\partial_t+v\partial_x)B=0$. We justify the differentiations and integrations by
parts in \Cref{lem:crossing-decay}, and we show in
\Cref{lem:crossing-monotone} that the functions in the resulting
covariances are monotone. After the proof, we deduce the shape of
$\Gamma_\nu'$ in \Cref{cor:Gamma-prime-shape}.

\begin{proposition}\label{prop:crossing}
Let $\nu\in\Pr([0,1])$, and suppose that $\alpha_\nu=m$ on an open
interval $I\subset(0,1)$, where $m>0$. Then $\Gamma_\nu\in C^3(I)$, and
for every $s\in I$,
\begin{equation}\label{eq:crossing}
 \Gamma_\nu''(s)=0\quad\Longrightarrow\quad\Gamma_\nu'''(s)>0.
\end{equation}
\end{proposition}

Throughout the rest of this section, $\nu,I,m$ are as in
\Cref{prop:crossing}. We use the time variable $t=\beta^2s$, which ranges
over $\beta^2I$, and suppress the time arguments. We write $u=u_\nu$,
$B=u_x$, and $C=u_{xx}$. The functions $z,Q,H$ are defined by
\eqref{eq:z-Q-H}, and $p$ and $V=mu-\log p$ are as in \Cref{prop:forward}. By \Cref{lem:constant-mass},
the functions $u$ and $\e^{-V}=p\e^{-mu}$ are smooth in space and time for
$s\in I$, and the latter solves the forward heat equation. We set
\[
 \omega=2pC^2,\qquad \mathcal Z=\int_0^\infty\omega\dd x,
 \qquad \rho(\dd x)=\frac{\omega(x)}{\mathcal Z}\dd x,
 \qquad \Phi=2z^2-mC,\qquad f=\E_\rho\Phi,
\]
so that $\rho$ is the law of $|X_s^\nu|$ reweighted by $C^2$ and
normalized. We also set
\[
 N=V_x+z-mB,\qquad K=N^2-V_{xx}+z^2+mC,\qquad
 \Psi=z(N+z)-Q.
\]
Since $(\log p)_x=mB-V_x=z-N$, the definitions give
\begin{equation}\label{eq:basic-derivatives}
\begin{gathered}
 C_x=-2Cz,\qquad z_x=Q+mC,\qquad Q_x=zQ-\tfrac12H_x,\\
 N_x=V_{xx}+Q,\qquad \omega_x=-(N+3z)\omega.
\end{gathered}
\end{equation}
Finally, we set $v=mB-z$ and $\mathcal D=\partial_t+v\partial_x$. The
first lemma computes the evolution along this velocity.

\begin{lemma}\label{lem:transport}
On $\beta^2I\times\R$, we have
\begin{equation}\label{eq:transport-identities}
 \mathcal DB=0,\qquad \mathcal DC=CQ,\qquad \mathcal Dz=zQ-\tfrac12Q_x,
 \qquad \mathcal D\Phi=Q\Phi+zH_x,
\end{equation}
together with $\mathcal D\log p=K/2-z^2-mC$ and
\begin{equation}\label{eq:weight-transport}
 \omega_t+(v\omega)_x=(K/2-Q-H)\omega.
\end{equation}
Further, for every $t\in\beta^2I$, we have the identity
\begin{equation}\label{eq:crossing-centered-identity}
 -\frac{(\omega z)_x}{\omega}=\Phi+\Psi.
\end{equation}
\end{lemma}

\begin{proof}
Differentiating the Parisi equation $u_t=-\frac12(u_{xx}+mu_x^2)$ in $x$
and using $C_x=-2Cz$, we find that $B_t=Cz-mBC=-vC$, which is the identity
$\mathcal DB=0$. Since $v_x=-Q$ by
\eqref{eq:basic-derivatives}, the operator $\mathcal D$ satisfies the
commutation rule $\mathcal D(g_x)=(\mathcal Dg)_x+Qg_x$. Applying this rule
to $B$ gives $\mathcal DC=CQ$ and hence $\mathcal D\log C=Q$. Since
$z=-\frac12(\log C)_x$, applying it to $\log C$ gives
$\mathcal Dz=zQ-\frac12Q_x$. Then
$\mathcal D\Phi=4z\mathcal Dz-m\mathcal DC=Q\Phi+z(2zQ-2Q_x)$, which equals
$Q\Phi+zH_x$ by \eqref{eq:basic-derivatives}.

For $\ell=\log p$, the forward equation $p_t=\frac12p_{xx}-m(Bp)_x$ gives
$\mathcal D\ell=\frac12\ell_{xx}+\frac12\ell_x^2-mC-z\ell_x$. Since
$\ell_x=z-N$ and $\ell_{xx}=mC-V_{xx}$ by \eqref{eq:basic-derivatives}, it
follows that $\mathcal D\log p=K/2-z^2-mC$. Using
$\mathcal D\log\omega=\mathcal D\log p+2Q$ and $v_x=-Q$, we find
\[
 \omega_t+(v\omega)_x=\omega(\mathcal D\log\omega+v_x)
 =\omega\Bigl(\frac K2-z^2-mC+Q\Bigr),
\]
which is \eqref{eq:weight-transport} because $H=z^2+mC-2Q$. Finally, we
find from the formulas for $z_x$ and $\omega_x$ in
\eqref{eq:basic-derivatives} that $-(\omega z)_x/\omega=(N+3z)z-Q-mC$, which equals $\Phi+\Psi$.
\end{proof}

\begin{remark}\label{rem:weighted-law}
We translate \Cref{lem:transport} into the magnetization coordinate of
\Cref{s:outline}. Since $B_x=C>0$, let $x(t,b)$ be the field satisfying
$B(t,x(t,b))=b$ for a fixed magnetization $b$.
The identity $\mathcal DB=0$ gives
$\partial_t x(t,b)=v(t,x(t,b))$, so the chain rule identifies
$(\mathcal DF)(t,x(t,b))$ with the partial time derivative
of $F(t,x(t,b))$ with $b$ held fixed, where $F$ is any smooth function. Let $r=pC$, viewed as a function of $t$ and
$B$. Since $\dd B=C\dd x$, the density of $M_s=u_x(s,X_s)$ is $p/C$, so
that $r=\chi^2(p/C)$ and $\omega\dd x=2r\dd B$. In particular, the image of $\rho$
under $u_x$ has density proportional to $r$ on $[0,1)$. Since
$\omega=2rC$, the identity \eqref{eq:weight-transport} together with
$\mathcal D\log C=Q$ and $v_x=-Q$ gives $\mathcal D\log r=K/2-Q-H$. Next,
$(\log r)_x=(\log p)_x+(\log C)_x=-(N+z)$, so that
$N+z=-\chi\,\partial_B\log r$. Differentiating once more with
$\partial_B=C^{-1}\partial_x$ and $C_x=-2Cz$, we find
$\frac12\chi^2r_{BB}/r=\frac12(N^2-z^2-N_x-z_x)$, and
\eqref{eq:basic-derivatives} shows that this equals $K/2-3Q-H$.
Comparing with $\mathcal D\log r$, we find that
$\mathcal Dr=\frac12\chi^2r_{BB}+2Qr$, which is \eqref{e:weighted-law}.
The formulas for $N_x$ and $z_x$ in \eqref{eq:basic-derivatives} give
$K=(N^2-N_x)+(z^2+z_x)$. Finally, since $\omega=2rC$ and
$\partial_x=C\,\partial_B$, the identity \eqref{eq:crossing-centered-identity}
reads $-\partial_B(\chi zr)/r=\Phi+\Psi$.
\end{remark}

The second lemma provides the decay that justifies differentiating under
the integral sign and integrating by parts.

\begin{lemma}\label{lem:crossing-decay}
For every compact interval $J\subset\beta^2I$, there exists a constant
$c_J>0$ such that
\[
 \omega+|\omega_t|+|(\Phi\omega)_t|\le c_J(1+x^2)\e^{x-x^2/(2\beta^2)}
 \qquad\text{for all } t\in J \text{ and } x\ge0.
\]
Further, the functions $v,z,\Phi,\Phi_x,\mathcal D\Phi,H$ are bounded on
$J\times[0,\infty)$, and $v=z=0$ at $x=0$.
\end{lemma}

\begin{proof}
The functions $z,Q,H,H_x$ are bounded by \Cref{prop:backward}, and since
$|B|\le1$ and $0<C\le1$ by \eqref{eq:curvature-bounds}, so are $v$, $\Phi$,
$\Phi_x=2z(2Q+3mC)$, and $\mathcal D\Phi=Q\Phi+zH_x$. Since $z$ and $B$ are
odd, we have $v=z=0$ at $x=0$. By \Cref{prop:forward}, we have
$V_x=x/t+W_x$ with $|W_x|\le1$ and $\lambda_s\le V_{xx}\le\lambda_s+1$ with
$\lambda_s=1/t$, so that $|N|\le c_J(1+x)$ and $|K|\le c_J(1+x^2)$ because
$J$ is a compact subset of $(0,\beta^2)$. By \eqref{eq:basic-derivatives} and
\eqref{eq:weight-transport}, we have
\[
 \omega_t=\bigl(v(N+3z)+K/2-H\bigr)\omega,
 \qquad (\Phi\omega)_t=(\mathcal D\Phi-v\Phi_x)\omega+\Phi\omega_t,
\]
so that $|\omega_t|+|(\Phi\omega)_t|\le c_J(1+x^2)\omega$ after increasing
$c_J$. Finally, \Cref{prop:forward}(b) gives
$\omega\le2p\le2\e^{\beta^2+x}g_t(x)$, and
$g_t(x)\le(2\pi\min J)^{-1/2}\e^{-x^2/(2\beta^2)}$ for $t\in J$.
\end{proof}

The third lemma combines the backward and forward inequalities into the
monotonicity of the functions in the covariances.

\begin{lemma}\label{lem:crossing-monotone}
For every $t\in\beta^2I$, the functions $\Phi,K,\Psi$ are strictly increasing
on $[0,\infty)$ and $H$ is nondecreasing there.
\end{lemma}

\begin{proof}
Recall from \Cref{prop:backward} that $Q\ge0$ and $H_x\ge0$ on
$[0,\infty)$, and from \Cref{prop:forward}(a) that $V_{xx}\ge\lambda_s>0$
and $V_{xxx}\le0$ there. The functions $z$ and $N$ are odd and therefore
vanish at $x=0$, and since $z_x=Q+mC>0$ and $N_x=V_{xx}+Q>0$ on
$[0,\infty)$ by \eqref{eq:basic-derivatives}, they are positive on
$(0,\infty)$.
Differentiating with the help of \eqref{eq:basic-derivatives}, we find
\begin{equation}\label{eq:crossing-monotone}
\begin{aligned}
 \Phi_x&=2z(2Q+3mC),\\
 K_x&=2N(V_{xx}+Q)-V_{xxx}+2zQ,\\
 \Psi_x&=(Q+mC)N+z(V_{xx}+2Q+2mC)+\tfrac12H_x.
\end{aligned}
\end{equation}
For $x>0$, every term on the right is nonnegative, and the first term in
each line is positive. This proves the lemma.
\end{proof}

\begin{proof}[Proof of \Cref{prop:crossing}]
By \Cref{lem:Gamma-derivatives}, we have $\Gamma_\nu\in C^3(I)$ and
$\Gamma_\nu''(s)=\beta^4\E[D_s^2-2mC_s^3]$. Since $u_{xxx}=C_x=-2Cz$, the
integrand $u_{xxx}^2-2mu_{xx}^3$ equals $2C^2\Phi$, and because this
function and the density $p$ are even, it follows that
\begin{equation}\label{eq:Gamma-as-integrals}
 \Gamma_\nu''(s)=2\beta^4\mathcal Z f.
\end{equation}
By \Cref{lem:crossing-decay}, we may differentiate $\mathcal Z$ and
$\int_0^\infty\Phi\omega\dd x$ under the integral sign, and the boundary
terms in the integrations by parts below vanish because the terms 
$v\omega,\Phi v\omega,\omega z,H\omega z$ vanish at $x=0$ and decay at
infinity. Since $\Phi_t=\mathcal D\Phi-v\Phi_x$,
\eqref{eq:weight-transport} implies that
\[
 (\Phi\omega)_t=(\mathcal D\Phi)\omega-(\Phi v\omega)_x
 +\Phi(K/2-Q-H)\omega.
\]
Integrating this identity and \eqref{eq:weight-transport} over
$[0,\infty)$ and differentiating the normalized expectation, we find by
\eqref{eq:transport-identities} that
\[
 f_t=\E_\rho[Q\Phi+zH_x]
       +\Cov_\rho(\Phi,K/2-Q-H).
\]
The  identity \eqref{eq:crossing-centered-identity} implies that
$\E_\rho\Psi=-\E_\rho\Phi=-f$. Integrating by parts, we also find
\[
 \E_\rho[zH_x]=\E_\rho[H(\Phi+\Psi)]
 =\Cov_\rho(H,\Phi)+\Cov_\rho(H,\Psi).
\]
We write $\E_\rho[Q\Phi]=(\E_\rho Q)f+\Cov_\rho(Q,\Phi)$ and substitute.
Then the $Q$ and $H$ covariances cancel, and we obtain
\begin{equation}\label{eq:positive-representation}
 f_t-(\E_\rho Q)f
 =\frac12\Cov_\rho(\Phi,K)+\Cov_\rho(H,\Psi)>0.
\end{equation}
For the sign, we use the identity
\[
 2\Cov_\rho(g,h)
 =\E[(g(Y)-g(Y'))(h(Y)-h(Y'))]
\]
for independent random variables $Y,Y'$ with law $\rho$.
By \Cref{lem:crossing-monotone}, the integrand is nonnegative for both
covariances in \eqref{eq:positive-representation}, and for the first one
it is positive whenever $Y\ne Y'$. Since $\rho$ has a positive density, $Y\ne Y'$
almost surely, and therefore the first covariance is positive. At a zero
of $\Gamma_\nu''$, we have $f=0$ by \eqref{eq:Gamma-as-integrals}.
Differentiating \eqref{eq:Gamma-as-integrals} in $s=t/\beta^2$, we then
conclude that $\Gamma_\nu'''=2\beta^6\mathcal Z f_t>0$.
\end{proof}

The identity \eqref{eq:positive-representation} holds at every time, not
only at a zero of $\Gamma_\nu''$, and it determines the shape of
$\Gamma_\nu'$ on the whole interval $I$.

\begin{corollary}\label{cor:Gamma-prime-shape}
Let $\nu,I,m$ be as in \Cref{prop:crossing}. There exists $s_*$ in the
closure of $I$ such that $\Gamma_\nu'$ is strictly decreasing on
$I\cap(0,s_*]$ and strictly increasing on $I\cap[s_*,1)$.
\end{corollary}

\begin{proof}
We write $\rho_r$ for the measure $\rho$ at time $r$. The function
$r\mapsto\E_{\rho_r}Q$ is continuous on $\beta^2I$ by dominated
convergence, since $Q$ is smooth and bounded by \Cref{prop:backward}, and
$\omega$ is continuous in time and dominated as in
\Cref{lem:crossing-decay}. We fix $t_0\in\beta^2I$ and set
\[
 g(t)=\exp\Bigl(-\int_{t_0}^t\E_{\rho_r}Q\dd r\Bigr)f(t),
 \qquad t\in\beta^2I.
\]
Then $g$ is differentiable, and \eqref{eq:positive-representation} gives
\[
 g'(t)=\exp\Bigl(-\int_{t_0}^t\E_{\rho_r}Q\dd r\Bigr)
 \bigl(f_t-(\E_\rho Q)f\bigr)>0.
\]
Since $g$ is strictly increasing, there exists $t_*$ in the closure of
$\beta^2I$ such that $g<0$ on $\beta^2I\cap(0,t_*)$ and $g>0$ on
$\beta^2I\cap(t_*,\beta^2)$. Since $\mathcal Z>0$,
\eqref{eq:Gamma-as-integrals} implies that $\Gamma_\nu''(s)$ has the sign of
$g(\beta^2s)$. With $s_*=t_*/\beta^2$, we conclude that $\Gamma_\nu''<0$ on
$I\cap(0,s_*)$ and $\Gamma_\nu''>0$ on $I\cap(s_*,1)$, which proves the
corollary.
\end{proof}

\section{Proof of the main results}\label{s:proof}

We first combine \Cref{cor:Gamma-prime-shape} with the optimality
conditions to exclude gaps in the support (Step~1 of the outline). We then prove
\Cref{thm:main,thm:atom} together: we identify the distribution function
in \Cref{lem:quotient}, bootstrap its regularity up to the terminal atom
in \Cref{lem:smooth-up-to-q}, and bound that atom in the final proof.

Throughout this section, $\beta>1$, $\mu=\mu_\beta$, and we write
\[
 \alpha=\alpha_\mu,\qquad u=u_\mu,\qquad X=X^\mu,\qquad \Gamma=\Gamma_\mu,
\]
and $C_s,D_s,A_s$ as in \eqref{eq:M-C-D-A}.

\begin{proposition}\label{prop:connected-support}
There exists $q_\beta\in(0,1)$ such that
$\supp\mu=[0,q_\beta]$.
\end{proposition}

\begin{proof}
Let $\Sigma=\supp\mu$ and $q_\beta=\max\Sigma$. If $q_\beta=0$, then
$\mu=\delta_0$ and the Parisi equation has the solution
$u(s,x)=\beta^2(1-s)/2+\log\cosh x$, so that
$\Gamma'(0)=\beta^2u_{xx}(0,0)^2=\beta^2>1$, contradicting
\eqref{eq:self-consistency} (cf.\ \cite{Toninelli2002}). If $q_\beta=1$, the same optimality
conditions would give $1=\Gamma(1)=\E\tanh^2(X_1)<1$. We conclude that
$q_\beta\in(0,1)$ (cf.\ \cite[Lemma~3.8]{JagannathTobasco2017}).

We first show that $0\in\Sigma$, following the proof of \cite[Theorem~1(i)]{AuffingerChenProperties}. Otherwise, $b=\min\Sigma>0$ and $\alpha=0$ on
$[0,b)$, and \Cref{lem:Gamma-derivatives} implies that
$\Gamma''(s)=\beta^4\E D_s^2>0$ for $0<s<b$. Indeed,
$X_s=\beta W_s$ has a positive Gaussian density, and
$u_{xxx}(s,\cdot)$ cannot vanish identically, since then
$u_{xx}(s,\cdot)$ would be a positive constant, contradicting $|u_x|\le1$.
By \eqref{eq:self-consistency} and $\Gamma(0)=0$, we then obtain
\[
 b=\Gamma(b)=\int_0^b\Gamma'(s)\dd s
 <b\Gamma'(b)\le b,
\]
a contradiction.

Suppose that $\Sigma\ne[0,q_\beta]$, and let $(a,b)$ be a component of
$[0,q_\beta]\setminus\Sigma$. Its endpoints belong to $\Sigma$, and
$\alpha=m$ on $[a,b)$, where $m=\mu([0,a])>0$ because $a\in\Sigma$
and $\mu((a,b))=0$. By \Cref{cor:Gamma-prime-shape} with $I=(a,b)$,
there exists $s_*\in[a,b]$ such that $\Gamma'$ is strictly decreasing on
$(a,s_*]\cap(a,b)$ and strictly increasing on $[s_*,b)\cap(a,b)$. Since
$\Gamma'$ is continuous on $[a,b]$ by \Cref{lem:Gamma-derivatives}, we
deduce from \eqref{eq:Gamma-prime} and \eqref{eq:self-consistency} that
\[
 \Gamma'(s)\le\max\{\Gamma'(a),\Gamma'(b)\}\le1
 \qquad\text{for } s\in[a,b].
\]
On the other hand, $\Gamma(a)=a$ and $\Gamma(b)=b$ by
\eqref{eq:self-consistency}, so that $\int_a^b(1-\Gamma'(s))\dd s=0$. Since the
integrand is continuous and nonnegative, it follows that $\Gamma'\equiv1$
on $[a,b]$. This contradicts the strict monotonicity of $\Gamma'$ on one of
the two intervals above, which has positive length. We conclude that
$\Sigma=[0,q_\beta]$.
\end{proof}

The quotient formula in the next lemma is the SK case of an identity in
the proof of \cite[Theorem~2(ii)]{AuffingerChenProperties}. Its physics
counterpart is \cite[(23)]{CrisantiRizzo2002}. For the rest of this section,
we write $q=q_\beta$ and set
\[
 \mathcal A(s)=\E D_s^2,\qquad \mathcal B(s)=\E C_s^3
 \qquad\text{for } s\in[0,1].
\]

\begin{lemma}\label{lem:quotient}
The functions $\mathcal A$ and $\mathcal B$ are continuous on $[0,1]$,
$\mathcal B>0$, and
\begin{equation}\label{eq:quotient}
 \alpha(s)=\frac{\mathcal A(s)}{2\mathcal B(s)}\quad(0\le s<q),
 \qquad
 \mu([0,q))=\frac{\mathcal A(q)}{2\mathcal B(q)}.
\end{equation}
In particular, $\widetilde\alpha=\mathcal A/(2\mathcal B)$ is a continuous
extension of $\alpha|_{[0,q)}$ to $[0,q]$, and $\mu(\{0\})=\alpha(0)=0$.
\end{lemma}

\begin{proof}
The functions $\mathcal A$ and $\mathcal B$ are continuous by
\Cref{lem:Gamma-derivatives}, and $\mathcal B>0$ because $u_{xx}>0$. By
\Cref{prop:connected-support} and \eqref{eq:self-consistency}, we have
$\Gamma(s)=s$ for all $s\in[0,q]$, and therefore $\Gamma'=1$ on $[0,q]$.
Then \eqref{eq:Gamma-prime-integral} implies that
\[
 \int_r^t\bigl(\mathcal A(s)-2\alpha(s)\mathcal B(s)\bigr)\dd s=0
 \qquad\text{for all } 0\le r<t\le q,
\]
so that $\mathcal A-2\alpha\mathcal B=0$ almost everywhere on $(0,q)$. Let
$s\in[0,q)$, and choose $s_n\downarrow s$ in this set of full measure, so
that $\alpha(s_n)=\mathcal A(s_n)/(2\mathcal B(s_n))$ for all $n$. As
$n\to\infty$, the left side tends to $\alpha(s)$ by the right continuity
of $\alpha$, and the right side tends to $\mathcal A(s)/(2\mathcal B(s))$
by continuity. This proves the first identity in \eqref{eq:quotient}. We
obtain the second by letting $s\uparrow q$, since $\alpha(s)\to\mu([0,q))$.
Finally, $X_0=0$ and $u_{xxx}(0,\cdot)$ is odd, so that $\mathcal A(0)=0$
and $\mu(\{0\})=\alpha(0)=0$.
\end{proof}

\begin{lemma}\label{lem:smooth-up-to-q}
The function $\widetilde\alpha$ belongs to $C^\infty([0,q])$, and so does
every polynomial moment $e_f$ of \Cref{lem:Gamma-derivatives}.
\end{lemma}

\begin{proof}
Since replacing $\alpha(q)$ by $\widetilde\alpha(q)$ changes $\alpha$ at a
single point only, we obtain from \eqref{eq:moment-integral} that
\[
 e_f(t)-e_f(r)
 =\beta^2\int_r^t
 \bigl(e_{f_0}(s)+\widetilde\alpha(s)e_{f_1}(s)\bigr)\dd s
\]
for every polynomial $f$ and all $0\le r<t\le q$.
For $k\ge0$, let $P(k)$ be the statement that $\widetilde\alpha$ and every
$e_f$ belong to $C^k([0,q])$ with one-sided derivatives at both
endpoints. By \Cref{lem:Gamma-derivatives,lem:quotient}, $P(0)$ holds.
If $P(k)$ holds, then the integrand above belongs to $C^k([0,q])$ for
every $f$, and therefore every $e_f$ belongs to $C^{k+1}([0,q])$. Since
$\mathcal A$ and $\mathcal B$ are such moments and $\mathcal B>0$, the
quotient $\widetilde\alpha=\mathcal A/(2\mathcal B)$ belongs to
$C^{k+1}([0,q])$ as well, which is $P(k+1)$. By induction, $P(k)$ holds for
every $k\ge0$.
\end{proof}

\begin{proof}[Proof of \Cref{thm:main,thm:atom}]
\emph{The density.} For $f=u_{xxx}^2$ and $f=u_{xx}^3$, we find from
\eqref{eq:f0-f1} that
$\mathcal A'=\beta^2\E[A_s^2-6\widetilde\alpha C_sD_s^2]$ and
$\mathcal B'=3\beta^2\E[C_sD_s^2-\widetilde\alpha C_s^4]$ on $[0,q]$.
Differentiating $\mathcal A=2\widetilde\alpha\mathcal B$, we obtain
\begin{equation}\label{eq:alpha-prime}
 \alpha'(s)
 =\frac{\beta^2\,\E\bigl[A_s^2-12\alpha(s)C_sD_s^2+6\alpha(s)^2C_s^4\bigr]}
 {2\,\E C_s^3},
 \qquad 0\le s<q,
\end{equation}
whose limit as $s\uparrow q$ is obtained by evaluating at $s=q$ with
$\alpha(s)$ replaced by $\mu([0,q))$.

We set $\rho_\beta=\widetilde\alpha'$, which belongs to $C^\infty([0,q])$
by \Cref{lem:smooth-up-to-q}, and we set $m=\widetilde\alpha(q)=\mu([0,q))$
and $c_\beta=1-m=\mu(\{q\})$. Since $\alpha$ is nondecreasing,
$\rho_\beta\ge0$, and by continuity this also holds at $q$. For
$0\le r<s<q$, we have
$\mu((r,s])=\alpha(s)-\alpha(r)=\int_r^s\rho_\beta(t)\dd t$ by the
fundamental theorem of calculus. Since also $\mu(\{0\})=0$, the two
measures agree on a generating semiring of $[0,q)$, which proves
\eqref{eq:main-decomposition}. Further,
$\int_0^q\rho_\beta(t)\dd t=\widetilde\alpha(q)-\widetilde\alpha(0)
=m=1-c_\beta$. Since $0\in\supp\mu$ and $q>0$, we have
$\mu([0,q/2))>0$, which gives $c_\beta<1$. It remains to bound this atom
away from zero.

\emph{The terminal atom.} As in the proof of
\cite[Theorem~4]{AuffingerChenProperties}, we evaluate the quotient
using the explicit terminal profile. Since $\alpha=1$ on $[q,1]$,
\Cref{lem:constant-mass}(i) implies that
$u(q,x)=\log\cosh x+\beta^2(1-q)/2$. We write $B=\tanh x$,
$C=\sech^2x$, and $V_-=V(q-)$. 
We normalize $2p(q,\cdot)C^2$ to a probability measure
$\rho_q$ on $[0,\infty)$. With the left mass
$m=\alpha(q-)$, the definitions become 
\[
 \begin{gathered}
 z=B,\qquad Q=(1-m)C,\qquad N=V_-'+(1-m)B,\\
 \Phi=2B^2-mC,\qquad \Psi=BV_-'+(2-m)B^2-(1-m)C.
 \end{gathered}
\]
Since $p(q,\cdot)$ is smooth by \Cref{prop:forward} and
$\log p(q,\cdot)=mu(q,\cdot)-V_-$ by \eqref{eq:V-def}, the formulas for
$z_x$ and $\omega_x$ in \eqref{eq:basic-derivatives} and the identity
\eqref{eq:crossing-centered-identity} hold at $q$ with $m=\alpha(q-)$. We
integrate \eqref{eq:crossing-centered-identity} against $\omega$ over
$[0,\infty)$. The boundary terms vanish because $z=0$ at $x=0$ and
$\omega z$ decays at infinity by \Cref{prop:forward}(b), so that
$\E_{\rho_q}(\Phi+\Psi)=0$. Since $D_q=-2C_q\tanh X_q$, the
second identity in \eqref{eq:quotient} says that
$\E[C_q^2\Phi(X_q)]=0$, and we conclude that
\begin{equation}\label{eq:atom-centering}
 \E_{\rho_q}\Phi=\E_{\rho_q}\Psi=0,
 \qquad 2\E_{\rho_q}B^2=m\E_{\rho_q}C.
\end{equation}
Let $\lambda=1/(\beta^2q)$. Since $V_-'(0)=0$ by evenness and
$V_-''\ge\lambda$ on $\R$ by \Cref{prop:forward}(a), and since
$\tanh x\le x$ for $x\ge0$, we have $V_-'\ge\lambda x\ge\lambda B$ on
$[0,\infty)$. Inserting this bound into the formula for $\Psi$, we have
\[
 0=\E_{\rho_q}\Psi
 \ge(2-m+\lambda)\E_{\rho_q}B^2-(1-m)\E_{\rho_q}C
 =\frac{m(4-m+\lambda)-2}{2}\E_{\rho_q}C.
\]
In the inequality, we used that $B\ge0$ on $[0,\infty)$, and the equality
follows from \eqref{eq:atom-centering}. Since $\E_{\rho_q}C>0$, this
proves \eqref{eq:atom-quadratic}. Since $\lambda>0$ and $0<m\le1$, this
bound implies that $m(4-m)<2$ and $\lambda m<2$. The first inequality means
that $(2-m)^2>2$, that is, $m<2-\sqrt2$. The second gives $m<2\beta^2q$. We conclude that
$c_\beta=1-m>\max\{\sqrt2-1,1-2\beta^2q\}>0$. Together with
\Cref{prop:connected-support}, this proves both theorems.
\end{proof}

As an example of \eqref{eq:alpha-prime}, at $s=0$ we have $X_0=0$, $D_0=0$, $\alpha(0)=0$, and
$C_0=u_{xx}(0,0)=1/\beta$ because $\Gamma'(0)=1$. Inserting these values into
\eqref{eq:alpha-prime}, we obtain $\alpha'(0)=\frac{\beta^5}{2}u_{xxxx}(0,0)^2$.
Formula \eqref{eq:alpha-prime} and these values at $s=0$ are the rigorous
counterparts of \cite[(24), (26), and (27)]{CrisantiRizzo2002}.

\begin{remark}\label{rem:AC-regularity}
Auffinger and Chen proved that if $(a,b)\subset\supp\mu_P$ for a Parisi
measure $\mu_P$ of a mixed $p$-spin model, then its distribution function
is infinitely differentiable on the half-open interval $[a,b)$
\cite[Theorem~2(ii)]{AuffingerChenProperties}. Their proof uses the first
identity in \eqref{eq:quotient}, followed by the same differentiation of
moments used above. Applied with $(a,b)=(0,q_\beta)$, as allowed by
\Cref{prop:connected-support}, their theorem gives smoothness on
$[0,q_\beta)$. The additional point here is that continuity of every
polynomial moment at $q_\beta$, together with the moment
identities and the quotient formula, allows the same
induction to give derivatives of every order up to that
endpoint.
\end{remark}

\bibliographystyle{plain}
\bibliography{frsb_final}
\end{document}